\documentclass[reqno, 10pt]{amsart}
\usepackage{amsmath,amscd,amsthm,amsxtra, textcomp}
\usepackage{mathtools}
\usepackage{epsfig,graphics,color,colortbl}
\usepackage{amssymb,latexsym}
\usepackage{mathabx}
\usepackage{mathrsfs,eucal,upgreek}
\usepackage[poly,all]{xy}
\usepackage{hyperref}
\usepackage{tikz-cd}
\usepackage{arydshln}
\usepackage{ytableau}
\usepackage[draft]{todonotes}
\usepackage{enumerate}

\makeatletter
\@namedef{subjclassname@2020}{\textup{2020} Mathematics Subject Classification}
\makeatother

\newtheorem{thm}{\bf Theorem}[section]
\newtheorem{df}[thm]{\bf Definition}
\newtheorem{prop}[thm]{\bf Proposition}

\newtheorem{lem}[thm]{\bf Lemma}
\newtheorem{rem}[thm]{\bf Remark}
\newtheorem{ex}[thm]{\bf Example}

\numberwithin{equation}{section}

\begin{document}
\title[Extremal weight crystals over affine Lie algebras of infinite rank]
{Extremal weight crystals over affine Lie algebras of infinite rank}
\author{TAEHYEOK HEO}

\address{Research Institute of Mathematics, Seoul National University, Seoul 08826, Korea}
\email{gjxogur123@snu.ac.kr}





\thanks{
This research was supported by Basic Science Research Program through the National Research Foundation of Korea(NRF) funded by the Ministry of Education (RS-2023-00241542).}

\begin{abstract}
We explain extremal weight crystals over affine Lie algebras of infinite rank using combinatorial models: 
a spinor model due to Kwon, and an infinite rank analogue of Kashiwara-Nakashima tableaux due to Lecouvey. 
In particular, we show that Lecouvey's tableau model is isomorphic to an extremal weight crystal of level zero.
Using these combinatorial models, we explain an algebra structure of the Grothendieck ring for a category consisting of some extremal weight crystals.
\end{abstract}

\maketitle
\setcounter{tocdepth}{1}
\tableofcontents

\section{Introduction}



Let $U_q(\mathfrak{g})$ be a quantum group associated with a Kac-Moody algebra $\mathfrak{g}$. 
For $\lambda \in P$, Kashiwara \cite{Kas94,Kas02} constructs an (integrable) extremal weight module $V(\lambda)$, which is a generalization of a highest weight module. 
It has been long studied (cf. \cite{Kas94,Kas02}) because it has a close connection to level-zero representations over quantum affine algebras. 
In particular, an extremal weight module $V(\lambda)$ for level-zero $\lambda \in P$ is isomorphic to a tensor product of fundamental representations (cf. \cite{Beck02,BN04,Na04}). 
As a combinatorial approach to understanding extremal weight modules, one makes use of the theory of a crystal base of an integrable $U_q(\mathfrak{g})$-module. 
Although it is a local basis at $q=0$, it still encodes a significant amount of information from the original module. 
It is proved in \cite{Kas94} that an extremal weight module $V(\lambda)$ has a crystal base $(L(\lambda), B(\lambda))$, 
and we call such $B(\lambda)$ an extremal weight crystal.


In this paper, we explain some properties of extremal weight crystals over affine Lie algebras of infinite rank in a combinatorial way. 
In general, Naito and Sagaki \cite{NS12} show that, for any $\lambda \in P$ with ${\rm lv}(\lambda) \geq 0$, 
\begin{equation} \label{eqn:intro decomp}
    B(\lambda) \ \cong \ B(\lambda^0) \otimes B(\lambda^+)
\end{equation}
for some $\lambda^0 \in E$ and $\lambda^+ \in P^+$ (Proposition \ref{prop:decomp into zero and posi}). 
Here, $E$ is a set of some level-zero weights and $P^+$ is the set of all dominant weights. 
Thus, we reduce the understanding of $B(\lambda)$ ($\lambda \in P$)
to the understanding of $B(\lambda^0)$ ($\lambda^0 \in E$) and $B(\lambda^+)$ ($\lambda^+ \in P^+$) and 
to finding the relations between $B(\lambda^0)$ and $B(\lambda^+)$. 
We will follow this argument in this paper.

Various combinatorial objects, including tableau models and path models (cf. \cite{NS12}), 
contributes to the combinatorial realizations of crystal bases.
As a classical result, when $\mathfrak{g}$ is of classical (finite) type, Kashiwara and Nakashima \cite{KN94} show that a set of ($\mathcal{I}_n^\mathfrak{g}$)-semistandard tableaux 
satisfying certain configuration conditions is isomorphic to the crystal base of an irreducible highest weight $U_q(\mathfrak{g})$-module. 
We call such a semistandard tableau a Kashiwara-Nakashima (KN for short) tableau.
As its generalization, Lecouvey \cite{Le09} introduces a `KN-like' tableau, which corresponds to an infinite rank analogue of KN tableaux. 
We simply call it a $\mathfrak{g}_\infty$-type KN tableau, where $\mathfrak{g}_\infty$ is an affine Lie algebra of infinite rank. 
One of our results is to explain a $\mathfrak{g}_\infty$-crystal structure of the set of $\mathfrak{g}_\infty$-type KN tableaux. 
In particular, we prove that it is isomorphic to $B(\lambda)$ for some $\lambda \in E$ as $\mathfrak{g}_\infty$-crystals (Theorem \ref{thm:level zero realization}). 
On the other hand, Kwon \cite{K15, K16} introduces another combinatorial model, which is called a spinor model, 
to describe $B(\lambda)$ for $\lambda \in P^+$ over Lie (super)algebras of type $BCD$ (see Theorem \ref{thm:spinor}). 

Using these combinatorial models, we explain the algebra structure of the Grothendieck ring $\mathcal{K}$ for a category $\mathcal{C}$ of some extremal weight crystals (Theorem \ref{thm:whole Grothendieck}). 
This result is a BCD-type analogue of the characterization of the Grothendieck ring 
given by \cite{K09b} when $\mathfrak{g}_\infty = \mathfrak{a}_{+\infty}$, and by \cite{K11} when $\mathfrak{g}_\infty = \mathfrak{a}_\infty$.
As similarly as \eqref{eqn:intro decomp}, we have 
\[ \mathcal{K} \,\cong\, \mathcal{K}^0 \otimes \mathcal{K}^+, \]
where $\mathcal{K}^0$ and $\mathcal{K}^+$ are subalgebras corresponding to $B(\lambda)$ for $\lambda \in E$ and $\lambda \in P^+_{\rm int}$, respectively. 
Here, $P^+_{\rm int}$ is a subset of $P^+$ with integer coordinates (see Section \ref{subsec:inf rank Lie alg}).
It is known in \cite{Le09} that $\mathcal{K}^0$ is isomorphic to the ring of symmetric functions. 
To describe $\mathcal{K}^+$ explicitly, we obtain the Jacobi-Trudi type character formula of $B(\lambda)$ for $\lambda \in P^+_{\rm int}$ (Proposition \ref{prop:dominant wt character}). 
When a Lie algebra $\mathfrak{g}$ is of finite type, we know the character formula of an irreducible highest weight $\mathfrak{g}$-crystal (see \cite{FH91,KT87,Macd} and references therein). 
In particular, we have well-known Jacobi-Trudi type formulas, which is a determinant formula whose entries are complete symmetric polynomials or elementary symmetric polynomials. 
In this paper, we obtain the Jacobi-Trudi type character formula of $B(\lambda)$ ($\lambda \in P^+_{\rm int}$) 
in terms of modified elementary symmetric functions $E_r$ based on the observation found in \cite{LZ06}.
Using the Jacobi-Trudi type formula, we show that $\mathcal{K}^+$ is isomorphic to the ring of power series (Theorem \ref{thm:Grothendieck positive}). 
The relation between $\mathcal{K}^0$ and $\mathcal{K}^+$ will be explained in Proposition \ref{prop:tensor decomp} and Theorem \ref{thm:whole Grothendieck}.

Unfortunately, our results do not cover the whole weight lattice when $\mathfrak{g}_\infty = \mathfrak{b}_\infty$ or $\mathfrak{d}_\infty$. 
Indeed, our statements hold only when the extremal weight has an even level when $\mathfrak{g}_\infty = \mathfrak{b}_\infty$ or $\mathfrak{d}_\infty$ (cf. Remark \ref{rem:half-integer weight}).
We have partial results for odd levels when $\mathfrak{g}_\infty = \mathfrak{b}_\infty$ or $\mathfrak{d}_\infty$, 
but it is still mysterious to find the entire results.

This paper is organized as follows. In Section \ref{sec:review} and \ref{sec:ext wt crystal}, we briefly review some necessary backgrounds:
affine Lie algebras of infinite rank and related notations (Section \ref{sec:review}), 
and extremal weight crystals and their properties (Section \ref{sec:ext wt crystal}). 
In Section \ref{sec:combinatorial realization}, we introduce combinatorial models to describe extremal weight crystals 
and show that they are combinatorial realizations of some extremal weight crystals. 
In Section \ref{sec:Jacobi-Trudi}, we define the Jacobi-Trudi type formula using modified elementary symmetric functions $E_r$ 
and show that it is the character of an extremal weight crystal. 
In Section \ref{sec:Grothendieck}, we explain the algebra structure of the Grothendieck ring $\mathcal{K}$ for the category of some extremal weight crystals.

{\bf Acknowledgement.} The author appreciates Jae-Hoon Kwon for his kind comments on the draft. 

\section{Preliminaries} \label{sec:review}
\subsection{Lie algebras of infinite rank} \label{subsec:inf rank Lie alg}
A Lie algebra of infinite rank means a Kac-Moody algebra associated with a generalized Cartan matrix of infinite rank. 
We say that a Lie algebra of infinite rank is of affine type if every finite order principal minor of the associated Cartan matrix is positive.
There are five (non-isomorphic) affine Lie algebras of infinite rank and 
these are referred to Lie algebras $\mathfrak{a}_{+\infty}, \mathfrak{a}_\infty, \mathfrak{b}_\infty, \mathfrak{c}_\infty$, and $\mathfrak{d}_\infty$ (cf. \cite{Kac}). 
The below diagrams are Dynkin diagrams corresponding to each affine Lie algebra of infinite rank.
\[ \mathfrak{a}_{+\infty} \quad:\quad \raisebox{-0.5em}{$\begin{tikzpicture}
    \node at (0, 0) {$\circ$};
    \node [above] at (0, 0.1) {$1$};
    \node at (1.5, 0) {$\circ$};
    \node [above] at (1.5, 0.1) {$2$};
    \node at (3, 0) {$\cdots$};
    \node at (4.5, 0) {$\circ$};
    \node [above] at (4.5, 0.1) {$n$};
    \node [right] at (6, 0) {$\cdots$};

    \draw (0.2+1.5*0, 0) -- (1.3+1.5*0, 0);
    \draw (0.2+1.5*1, 0) -- (1.2+1.5*1, 0);
    \draw (0.3+1.5*2, 0) -- (1.3+1.5*2, 0);
    \draw (0.3+1.5*3, 0) -- (1.3+1.5*3, 0);
\end{tikzpicture}$} \]

\[ \mathfrak{a}_\infty \quad:\quad \raisebox{-0.5em}{$\begin{tikzpicture}[xscale=0.9]
    \node at (0, 0) {$\circ$};
    \node [above] at (0, 0.1) {$0$};
    \node at (1.5, 0) {$\circ$};
    \node [above] at (1.5, 0.1) {$1$};
    \node at (-1.5, 0) {$\circ$};
    \node [above] at (-1.5, 0.1) {$-1$};
    \node at (3, 0) {$\cdots$};
    \node at (-3, 0) {$\cdots$};
    \node at (4.5, 0) {$\circ$};
    \node [above] at (4.5, 0.1) {$n$};
    \node at (-4.5, 0) {$\circ$};
    \node [above] at (-4.5, 0.1) {$-n$};
    \node [right] at (4.5, 0) {$\cdots$};
    \node [left] at (-4.5, 0) {$\cdots$};

    \draw (0.2+1.5*0, 0) -- (1.3+1.5*0, 0);
    \draw (0.2-1.5*1, 0) -- (1.3-1.5*1, 0);
    \draw (0.2+1.5*1, 0) -- (1.2+1.5*1, 0);
    \draw (0.3-1.5*2, 0) -- (1.3-1.5*2, 0);
    \draw (0.3+1.5*2, 0) -- (1.3+1.5*2, 0);
    \draw (0.2-1.5*3, 0) -- (1.2-1.5*3, 0);
\end{tikzpicture}$} \]

\[ \mathfrak{b}_\infty \quad:\quad \raisebox{-0.5em}{$\begin{tikzpicture}
    \node at (0, 0) {$\circ$};
    \node [above] at (0, 0.1) {$0$};
    \node at (1.5, 0) {$\circ$};
    \node [above] at (1.5, 0.1) {$1$};
    \node at (3, 0) {$\circ$};
    \node [above] at (3, 0.1) {$2$};
    \node at (4.5, 0) {$\circ$};
    \node [above] at (4.5, 0.1) {$3$};
    \node at (6, 0) {$\cdots$};
    \node at (7.5, 0) {$\circ$};
    \node [above] at (7.5, 0.1) {$n-1$};
    \node at (9, 0) {$\cdots$};
    
    \draw (0.2+1.5*0, 0.05) -- (1.3+1.5*0, 0.05);
    \draw (0.2+1.5*0, -0.05) -- (1.3+1.5*0, -0.05);
    \draw (0.17, 0) -- (0.3, 0.2);
    \draw (0.17, 0) -- (0.3, -0.2);
    \draw (0.2+1.5*1, 0) -- (1.3+1.5*1, 0);
    \draw (0.2+1.5*2, 0) -- (1.3+1.5*2, 0);
    \draw (0.2+1.5*3, 0) -- (1.2+1.5*3, 0);
    \draw (0.3+1.5*4, 0) -- (1.3+1.5*4, 0);
    \draw (0.2+1.5*5, 0) -- (1.2+1.5*5, 0);
\end{tikzpicture}$} \]

\[ \mathfrak{c}_\infty \quad:\quad \raisebox{-0.5em}{$\begin{tikzpicture}
    \node at (0, 0) {$\circ$};
    \node [above] at (0, 0.1) {$0$};
    \node at (1.5, 0) {$\circ$};
    \node [above] at (1.5, 0.1) {$1$};
    \node at (3, 0) {$\circ$};
    \node [above] at (3, 0.1) {$2$};
    \node at (4.5, 0) {$\circ$};
    \node [above] at (4.5, 0.1) {$3$};
    \node at (6, 0) {$\cdots$};
    \node at (7.5, 0) {$\circ$};
    \node [above] at (7.5, 0.1) {$n-1$};
    \node at (9, 0) {$\cdots$};
    
    \draw (0.2+1.5*0, 0.05) -- (1.3+1.5*0, 0.05);
    \draw (0.2+1.5*0, -0.05) -- (1.3+1.5*0, -0.05);
    \draw (1.33, 0) -- (1.2, 0.2);
    \draw (1.33, 0) -- (1.2, -0.2);
    \draw (0.2+1.5*1, 0) -- (1.3+1.5*1, 0);
    \draw (0.2+1.5*2, 0) -- (1.3+1.5*2, 0);
    \draw (0.2+1.5*3, 0) -- (1.2+1.5*3, 0);
    \draw (0.3+1.5*4, 0) -- (1.3+1.5*4, 0);
    \draw (0.2+1.5*5, 0) -- (1.2+1.5*5, 0);
\end{tikzpicture}$} \]

\[ \mathfrak{d}_\infty \quad:\quad \raisebox{-2em}{$\begin{tikzpicture}
    \node at (0, 0.5) {$\circ$};
    \node [above] at (0, 0.6) {$0$};
    \node at (0, -0.5) {$\circ$};
    \node [above] at (0, -0.4) {$1$};
    \node at (1.5, 0) {$\circ$};
    \node [above] at (1.5, 0.1) {$2$};
    \node at (3, 0) {$\circ$};
    \node [above] at (3, 0.1) {$3$};
    \node at (4.5, 0) {$\circ$};
    \node [above] at (4.5, 0.1) {$4$};
    \node at (6, 0) {$\cdots$};
    \node at (7.5, 0) {$\circ$};
    \node [above] at (7.5, 0.1) {$n$};
    \node at (9, 0) {$\cdots$};
    
    \draw (0.2+1.5*0, 0.5) -- (1.3+1.5*0, 0.1);
    \draw (0.2+1.5*0, -0.5) -- (1.3+1.5*0, -0.1);
    \draw (0.2+1.5*1, 0) -- (1.3+1.5*1, 0);
    \draw (0.2+1.5*2, 0) -- (1.3+1.5*2, 0);
    \draw (0.2+1.5*3, 0) -- (1.2+1.5*3, 0);
    \draw (0.3+1.5*4, 0) -- (1.3+1.5*4, 0);
    \draw (0.2+1.5*5, 0) -- (1.2+1.5*5, 0);
\end{tikzpicture}$} \]

In this paper, we focus on providing results for $\mathfrak{g} = \mathfrak{b}_\infty, \mathfrak{c}_\infty$, or $\mathfrak{d}_\infty$. 
The corresponding results to ours can be found in \cite{K09b} when $\mathfrak{g} = \mathfrak{a}_{+\infty}$ 
and in \cite{K11} when $\mathfrak{g} = \mathfrak{a}_\infty$.
We use the following notations for affine Lie algebras of infinite rank.
\begin{itemize}
    \item $I = \mathbb{Z}_+$ : the index set (the set of nonnegative integers)
    \item $\Pi = \{\, \alpha_i \,|\, i \in I \,\}$ : the set of simple roots
    \item $\Pi^\vee = \{\, \alpha_i^\vee \,|\, i \in I \,\}$ : the set of simple coroots
    \item $\{\, \Lambda_i^\mathfrak{g} \,|\, i \in I \}$ : the set of fundamental weights
    \item $\displaystyle P = \mathbb{Z}\Lambda_0^\mathfrak{g} \,\oplus\, \bigoplus_{i=1}^\infty \mathbb{Z}\epsilon_i$ : the weight lattice, where $\epsilon_i$ is the $i$-th standard basis
    \item $P^+$ : the set of dominant weights, \qquad $\displaystyle E \,=\, \bigoplus_{i=1}^\infty\, \mathbb{Z}\epsilon_i \ \subseteq P$
    \item $\displaystyle Q = \bigoplus_{i \in I} \mathbb{Z}\alpha_i$ : the root lattice
    \item $\displaystyle Q^+ = \bigoplus_{i \in I} \mathbb{Z}_+ \alpha_i$
    \item $W = \langle s_i \,(i \in I) \,|\, s_i^2 = 1, (s_is_j)^{m_{ij}} = 1 \rangle$ : the Weyl group
\end{itemize}
In particular, we take simple roots $\alpha_i$ as follows and we obtain the corresponding fundamental weights.
\begin{eqnarray*}
    \mathfrak{c}_\infty & & \alpha_0 = -2\epsilon_1, \quad \alpha_i = \epsilon_i - \epsilon_{i+1} \ (i \geq 1) \\
    && \Lambda_0^\mathfrak{c} = -\sum_{i=1}^\infty \epsilon_i, \quad \Lambda_i^\mathfrak{c} = \Lambda_0^\mathfrak{c} + (\epsilon_1 + \dots + \epsilon_i) \ (i \geq 1) \\
    \mathfrak{b}_\infty & & \alpha_0 = -\epsilon_1, \quad \alpha_i = \epsilon_i - \epsilon_{i+1} \ (i \geq 1) \\
    && \Lambda_0^\mathfrak{b} = -\frac{1}{2}\sum_{i=1}^\infty \epsilon_i, \quad \Lambda_i^\mathfrak{b} = 2\Lambda_0^\mathfrak{b} + (\epsilon_1 + \dots + \epsilon_i) \ (i \geq 1) \\
    \mathfrak{d}_\infty & & \alpha_0 = -\epsilon_1 - \epsilon_2, \quad \alpha_i = \epsilon_i - \epsilon_{i+1} \ (i \geq 1) \\
    && \Lambda_0^\mathfrak{d} = -\frac{1}{2}\sum_{i=1}^\infty \epsilon_i, \quad \Lambda_1^\mathfrak{d} = \Lambda_0^\mathfrak{d} + \epsilon_1, \quad \Lambda_i^\mathfrak{d} = 2\Lambda_0^\mathfrak{d} + (\epsilon_1 + \dots + \epsilon_i) \ (i \geq 2)
\end{eqnarray*}
Moreover, when we write an element $\lambda = \sum m_i\epsilon_i \in P$, we have $\lambda \in P^+$ if and only if $\lambda$ satisfies that $m_i -m_{i+1} \in \mathbb{Z}$ for all $i \geq 1$ and 
\begin{itemize}
    \item when $\mathfrak{g} = \mathfrak{c}_\infty$, $0 \geq m_1 \geq m_2 \geq m_3 \geq \dots$ with $m_i \in \mathbb{Z}$,
    \item when $\mathfrak{g} = \mathfrak{b}_\infty$, $0 \geq m_1 \geq m_2 \geq m_3 \geq \dots$ with $m_i \in \mathbb{Z}$ or $m_i \in \frac{1}{2}+\mathbb{Z}$,
    \item when $\mathfrak{g} = \mathfrak{d}_\infty$, $0 \geq |m_1| \geq m_2 \geq m_3 \geq \dots$ with $m_i \in \mathbb{Z}$ or $m_i \in \frac{1}{2}+\mathbb{Z}$.
\end{itemize}
We often write the expression $\mathfrak{g} = \mathfrak{b}, \mathfrak{c}, \mathfrak{d}$ 
when we consider only the type of a Lie algebra $\mathfrak{g}$ without specifying its rank.

The level of $\lambda \in P$, denoted by ${\rm lv}(\lambda)$, is defined to be the value $\langle \lambda, K \rangle$, 
where $K$ is the canonical central element of affine Lie algebras of infinite rank (cf. \cite[Section 7.12]{Kac}). 
In particular, when we write $\lambda = \sum_i m_i\epsilon_i$, we have $m_i = -\frac{\epsilon_\mathfrak{g}}{2}{\rm lv}(\lambda)$ for some sufficiently large $i$, where
\[ \epsilon_\mathfrak{g} = \begin{cases}
    2 & \mbox{if } \mathfrak{g} = \mathfrak{c}, \\
    1 & \mbox{if } \mathfrak{g} = \mathfrak{b}\mbox{ or }\mathfrak{d}.
\end{cases} \]
It is clear that the level of dominant weights is nonnegative and $0 \in P$ is the unique level-zero dominant weight. 





Let $e_i, f_i$ ($i \in I$) be the Chevalley generators of $\mathfrak{g}_\infty$ ($\mathfrak{g} = \mathfrak{b}, \mathfrak{c}$, or $\mathfrak{d}$). 
For integers $n \geq 2$, let $\mathfrak{g}_n$ be the Lie subalgebra of $\mathfrak{g}_\infty$ generated by $e_i, f_i$ for $i = 0, 1, \dots, n-1$. 
For each $\mathfrak{g} = \mathfrak{b}, \mathfrak{c}, \mathfrak{d}$ with $n \geq 2$ ($n \geq 4$ when $\mathfrak{g} = \mathfrak{d}$), 
$\mathfrak{g}_n$ is isomorphic to $\mathfrak{so}_{2n+1}, \mathfrak{sp}_{2n}, \mathfrak{so}_{2n}$, respectively. 
This paper does not consider $\mathfrak{d}_2$ and $\mathfrak{d}_3$ (which is actually isomorphic to $\mathfrak{a}_3$) 
although the condition we will address seems to encompass two cases.

Now, we introduce some notations related to partitions. Let $\mathcal{P}$ be the set of partitions. 
Denote by $\ell(\lambda)$ the length of $\lambda \in \mathcal{P}$, by $|\lambda|$ the size of $\lambda \in P$, and 
let $\mathcal{P}_n \subseteq \mathcal{P}$ ($n \in \mathbb{Z}_+$) be the set of partitions $\lambda$ with $\ell(\lambda) \leq n$. 
Let $\lambda' = (\lambda_1', \lambda_2', \dots)$ be the conjugate of $\lambda \in \mathcal{P}$. 
For $\lambda \in \mathcal{P}_n$, set
\[ \varpi_\lambda = \sum_{i=1}^n \lambda_i \epsilon_i \in E. \]
For simplicity, we write $\varpi_i = \varpi_{(1^i)}$ for $i \geq 1$ and $\varpi_0 = 0$. 

For even $\ell \geq 2$, let $G_\ell$ be one of the following algebraic groups: ${\rm Sp}_\ell$, ${\rm Pin}_\ell$, or ${\rm O}_\ell$. Let
\begin{eqnarray*}
    && \mathscr{P}({\rm Sp}_\ell) = \mathcal{P}_{\frac{\ell}{2}}, \qquad
    \mathscr{P}({\rm Pin}_\ell) = \mathcal{P}_{\frac{\ell}{2}}, \\
    && \mathscr{P}({\rm O}_\ell) = \{ \lambda \in \mathcal{P}_\ell \,|\, \lambda_1' + \lambda_2' \leq \ell \,\}
\end{eqnarray*}
and, define
\[ \mathscr{P}(G) = \left. \left\{ \left( \lambda, \ell \right) \,\right|\, \ell \in \mathbb{N}, \lambda \in \mathscr{P}(G_{2\ell}) \right\}, \]
which plays a role of the disjoint union of $\mathscr{P}(G_\ell)$ for $G = {\rm Sp}$, ${\rm Pin}$, and ${\rm O}$. 
We suppose that a Lie algebra $\mathfrak{g}$ corresponds to an algebraic group $G$ (and vice versa) as follows:
$(\mathfrak{g}, G) = (\mathfrak{b}, {\rm Pin}), (\mathfrak{c}, {\rm Sp})$, and $(\mathfrak{d}, {\rm O})$.
We put
\begin{itemize}
    \item $\Pi_i^\mathfrak{c} = \Lambda_i^\mathfrak{c}$ \ ($i \geq 0$)
    \item $\Pi_0^\mathfrak{b} = 2\Lambda_0^\mathfrak{b}$, $\Pi_i^\mathfrak{b} = \Lambda_i^\mathfrak{b}$ \ ($i \geq 1$)
    \item $\Pi_0^\mathfrak{d} = 2\Lambda_0^\mathfrak{d}$, $\overline{\Pi}_0^\mathfrak{d} = 2\Lambda_1^\mathfrak{d}$, 
    $\Pi_1^\mathfrak{d} = \Lambda_0^\mathfrak{d} + \Lambda_1^\mathfrak{d}$, and $\Pi_i^\mathfrak{d} = \Lambda_i^\mathfrak{d}$ \ ($i \geq 2$)
\end{itemize}
and
\[ \Pi^\mathfrak{g}(\lambda, \ell) = \ell\,\Pi_0^\mathfrak{g} + \varpi_{\lambda'} \in P^+ \]
for $(\lambda, \ell) \in \mathscr{P}(G)$. For $(\lambda, \ell) \in \mathscr{P}(G)$ with $\ell(\lambda) = t$, we have
\begin{equation*}
    \Pi^\mathfrak{g}(\lambda, \ell) = \begin{cases}
        \Pi^\mathfrak{g}_{\lambda_1} + \dots + \Pi^\mathfrak{g}_{\lambda_\ell} & \mbox{if }t \leq \ell, \\
        \Pi^\mathfrak{d}_{\lambda_1} + \dots + \Pi^\mathfrak{d}_{\lambda_{2\ell-t}} + (t-\ell) \overline{\Pi}_0^\mathfrak{d} & \mbox{if } t > \ell.
    \end{cases}
\end{equation*}
The condition $(\lambda, \ell) \in \mathscr{P}(G)$ with $\ell(\lambda) > \ell$ holds only when $(\mathfrak{g}, G) = (\mathfrak{d}, {\rm O})$. 
In this case, we can check
\begin{equation} \label{eqn:well-defined partition}
    \lambda_{2\ell-t} \geq 1 \mbox{ and } (\lambda_1, \dots, \lambda_{2\ell-t}, 1^{t-\ell}) \mbox{ is a well-defined partition.}
\end{equation}

Let
\[ P^+_{\rm int} = \{ \,\Pi^\mathfrak{g}(\lambda, \ell) \,|\, (\lambda, \ell) \in \mathscr{P}(G) \,\} \subseteq P^+. \]
Note that $P^+_{\rm int} = P^+$ when $\mathfrak{g} = \mathfrak{c}_\infty$ and,
$P^+_{\rm int} \subsetneq P^+$ is the set of dominant weights with a positive even level when $\mathfrak{g} = \mathfrak{b}_\infty$ or $\mathfrak{d}_\infty$. 

\begin{rem} \label{rem:half-integer weight}
    The correspondences $(\mathfrak{g}, G)$ are closely related to the dual pair, which is introduced by Howe (cf. \cite{H89,H95,W99}).
    In general, let $\mathscr{F}$ be the fermionic Fock space given in \cite{W99}.
    Then we have the following $(\mathfrak{g}, G_\ell)$-duality on $\mathscr{F}^{\otimes \frac{\ell}{2}}$ for even $\ell \geq 2$.
    \[ \mathscr{F}^{\otimes \frac{\ell}{2}} \cong \bigoplus_{\lambda \in \mathscr{P}(G_\ell)} L(\mathfrak{g}, \Lambda^\mathfrak{g}(\lambda, \ell)) \otimes V^\lambda_{G_\ell} \]
    Here, $L(\mathfrak{g}, \Lambda^\mathfrak{g}(\lambda, \ell))$ is the irreducible $\mathfrak{g}$-modules 
    with highest weight $\Lambda^\mathfrak{g}(\lambda, \ell) = \frac{\ell}{\epsilon_\mathfrak{g}} \Lambda_0^\mathfrak{g} + \varpi_{\lambda'} \in P^+$, and
    $V^\lambda_{G_\ell}$ is the finite-dimensional irreducible representation of $G_\ell$ corresponding to $\lambda$. 
    In addition, the $(\mathfrak{g}, G_\ell)$-duality theorem holds for odd $\ell \geq 1$, where the corresponding dual pairs are $(\mathfrak{b}_\infty, {\rm Spin}_\ell)$ and $(\mathfrak{d}_\infty, {\rm O}_\ell)$ 
    and $\mathscr{F}^{\otimes (k + \frac{1}{2})}$ ($k \in \mathbb{Z}_+$) is the Fock space of $k$ pairs of fields with an extra neutral field (cf. \cite{W99}).

    In general, we know that $\mathscr{P}(G_\ell)$ parametrizes the set of all dominant weights of $\mathfrak{g}$ with level $\ell$, 
    which implies that $P^+ = \{ \Lambda^\mathfrak{g}(\lambda, \ell) \,|\, \ell \geq 1, \lambda \in \mathscr{P}(G_\ell) \} \cup \{ 0 \}$.
    In particular, we have $\Lambda^\mathfrak{g}(\lambda, 2\ell) = \Pi^\mathfrak{g}(\lambda, \ell)$ for $(\lambda, \ell) \in \mathscr{P}(G)$.
\end{rem}

\subsection{Crystals} \label{subsec:crystal}
We recall the notion of (abstract) crystals (see \cite{HongKang,Kas95} for details). 
Let $\mathfrak{g}$ be a symmetrizable Kac-Moody algebra associated with a generalized Cartan matrix $A = (a_{ij})_{i, j \in I}$. 
A $\mathfrak{g}$-crystal (or simply crystal) is a set $B$ equipped with maps ${\rm wt} : B \to P$, $\widetilde{e}_i, \widetilde{f}_i : B \to B \cup \{ {\bf 0} \}$, 
and $\varepsilon_i, \varphi_i : B \to \mathbb{Z} \cup \{ -\infty \}$ for $i \in I$ satisfying certain axioms. 
Here, ${\bf 0}$ is a formal symbol and $-\infty$ is another formal symbol with $-\infty + (-\infty) = -\infty$, and 
$-\infty < n$, $-\infty + n = -\infty$ for any $n \in \mathbb{Z}$.
For example, let $V(\lambda)$ be the irreducible highest weight $U_q(\mathfrak{g})$-module of highest weight $\lambda \in P^+$. 
There exists the (unique) crystal base $B(\lambda)$ associated with $V(\lambda)$ (see \cite{Kas91}) such that
\[ \varepsilon_i(b) = \max\{\, k \in \mathbb{Z}_+ \,|\ \widetilde{e}_i^kb \neq 0 \,\}, \quad \varphi_i(b) = \max\{\, k \in \mathbb{Z}_+ \,|\, \widetilde{f}_i^kb \neq 0 \,\}. \]
Note that the crystal base $B(\lambda)$ still satisfies the axiom for crystals.

We also review the notion related to semistandard tableaux (cf. \cite{Ful}), which is frequently chosen to describe crystals.
A semistandard tableau of shape $\lambda \in \mathcal{P}$ is a filling of $\lambda$ whose entries are weakly increasing from left to right in each row and strictly increasing from top to bottom in each column. 
Let $SST(\lambda)$ be the set of semistandard tableaux of shape $\lambda$ with letters in $\mathbb{N}$.
In particular, we call a semistandard tableau with a single column a column tableau.
In this case, set the height of a column tableau $C$, denoted by ${\rm ht}(C)$, to be the number of its cells.
When we want to emphasize a letter set $X$, we record it as a prefix or subscript. 
In particular, when $X = \{ 1, \dots, n \}$, we record $n$ instead of $X$, 
i.e., $SST_n(\lambda)$ is the set of $n$-semistandard tableaux of shape $\lambda \in \mathcal{P}$.

It is well-known that the set $SST_n(\lambda)$ ($\lambda \in \mathcal{P}_n$) is a $\mathfrak{gl}_n$-crystal and is isomorphic to $B(\varpi_\lambda)$ as $\mathfrak{gl}_n$-crystals. 
Note that its highest weight vector is the tableau $T_\lambda$ of shape $\lambda$ such that
\begin{equation} \label{eqn:ext wt vec for partition}
    T_\lambda(i, j) = i, \quad \mbox{for all $1 \leq i \leq \ell(\lambda)$ and $1 \leq j \leq \lambda_i$},
\end{equation}
where $T(i, j)$ is the $(i, j)$-entry of a given tableau $T$.

The tensor product $B_1 \otimes B_2$ of crystals $B_1$ and $B_2$ is the set $B_1 \times B_2$ equipped with the following maps: 
for $i \in I$ and $b_i \in B_i$ ($i = 1, 2$),
\begin{itemize}
    \item ${\rm wt}(b_1 \otimes b_2) = {\rm wt}(b_1) + {\rm wt}(b_2)$
    \item $\varepsilon_i(b_1 \otimes b_2) = \max\{\, \varepsilon(b_1),\ \varepsilon_i(b_2) - \langle {\rm wt}(b_1), \alpha_i^\vee \rangle \,\}$
    \item $\varphi_i(b_1 \otimes b_2) = \max\{\, \varphi(b_1) + \langle {\rm wt}(b_2), \alpha_i^\vee \rangle,\ \varphi_i(b_2) \,\}$
    \item $\widetilde{e}_i(b_1 \otimes b_2) = \begin{cases}
        \widetilde{e}_ib_1 \otimes b_2, & \mbox{if } \varphi_i(b_1) \geq \varepsilon_i(b_2), \\
        b_1 \otimes \widetilde{e}_ib_2, & \mbox{if } \varphi_i(b_1) < \varepsilon_i(b_2)
    \end{cases}$
    \item $\widetilde{f}_i(b_1 \otimes b_2) = \begin{cases}
        \widetilde{f}_ib_1 \otimes b_2, & \mbox{if } \varphi_i(b_1) > \varepsilon_i(b_2), \\
        b_1 \otimes \widetilde{f}_ib_2, & \mbox{if } \varphi_i(b_1) \leq \varepsilon_i(b_2).
    \end{cases}$
\end{itemize}
Here, we write $(b_1, b_2) \in B_1 \otimes B_2$ as $b_1 \otimes b_2$, and 
assume that ${\bf 0} \otimes b_2 = b_1 \otimes {\bf 0} = {\bf 0}$ for any $b_1 \in B_1$ and $b_2 \in B_2$.

For a crystal $B$, denote by $C(b)$ the connected component of $b \in B$. 
For two crystals $B_1$ and $B_2$, we say that $b_1 \in B_1$ and $b_2 \in B_2$ are crystal equivalent (to each other) 
if there exists an isomorphism $\phi : C(b_1) \to C(b_2)$ of crystals sending $b_1$ to $b_2$.

\section{Extremal weight crystals} \label{sec:ext wt crystal}
\subsection{Extremal weight crystals} \label{subsec:ext wt crystal}
We review properties of extremal weight crystals (see \cite{Kas94,Kas02} and references therein).
\begin{df}
    Let $M$ be an integrable $U_q(\mathfrak{g})$-module. A weight vector $v \in M$ of weight $\lambda \in P$ is said to be extremal 
    if there exists a family $\{ v_w \}_{w \in W}$ of weight vectors in $M$ such that 
    if we put $m = \langle h_i, w\lambda \rangle$ for $i \in I$ and $w \in W$, then
    \begin{itemize}
        \item $v_e = v$,
        \item we have $e_iv_w = 0$ and $f_i^{(m)}v_w = v_{s_iw}$ when $m \geq 0$,
        \item we have $f_iv_w = 0$ and $e_i^{(-m)}v_w = v_{s_iw}$ when $m \leq 0$.
    \end{itemize}
\end{df}

For $\lambda \in P$, Kashiwara \cite{Kas94} constructs an (integrable) extremal weight $U_q(\mathfrak{g})$-module $V(\lambda)$ 
with an extremal weight vector $v_\lambda \in V(\lambda)$ of weight $\lambda$. 
In particular, $V(\lambda)$ is a highest (resp. lowest) weight module when $\lambda \in P^+$ (resp. $\lambda \in -P^+$). 
However, $V(\lambda)$ is neither a highest weight module nor a lowest weight module, in general.
Nonetheless, it still has a significant property; it has a crystal base.
\begin{thm}[\cite{Kas94}]
    For $\lambda \in P$, the extremal weight module $V(\lambda)$ has a crystal base $(L(\lambda), B(\lambda))$.
\end{thm}
We call such $B(\lambda)$ an extremal weight crystal and 
let $b_\lambda \in B(\lambda)$ be the image of $v_\lambda$ under the canonical projection $V(\lambda) \twoheadrightarrow  B(\lambda)$. 
The extremal weight module $V(\lambda)$ and its crystal $B(\lambda)$ have a close connection to the $W$-orbit of the weight $\lambda \in P$. 
\begin{prop}[\cite{Kas94}, {\cite[Proposition 3.9]{NS12}}] \label{prop:W-orbit isomorphism}
    For $w \in W$, we have isomorphisms $V(\lambda) \cong V(w\lambda)$ and $B(\lambda) \cong B(w\lambda)$. 
    Moreover, when $\mathfrak{g}$ is an affine Lie algebra of infinite rank, for $\lambda, \mu \in P$, 
    we have $\lambda \in W\mu$ if and only if $B(\lambda) \cong B(\mu)$.
\end{prop}

On the other hand, we can find an extremal weight crystal in a certain tensor product of two extremal weight crystals.
\begin{lem}[{\cite[Lemma 8.3.1]{Kas94}}] \label{lem:summand in tensor}
    Let $\lambda \in P^+$ and $\mu \in P$. The connected component $C(b_{\lambda+\mu})$ in $B(\lambda + \mu)$ 
    is isomorphic to the connected component $C(b_\lambda \otimes b_\mu)$ in $B(\lambda) \otimes B(\mu)$. 
    In particular, the isomorphism sends $b_{\lambda+\mu}$ to $b_\lambda \otimes b_\mu$.
\end{lem}

From now on, we assume that all Lie algebras in this article are affine Lie algebras of infinite rank without otherwise stated. 
In particular, we use the notation $\mathfrak{g}_\infty$ to emphasize the infinite rank.

\begin{prop}[{\cite[Proposition 3.6]{NS12}}] \label{prop:conn}
    For $\lambda \in P$, $B(\lambda)$ is connected.
\end{prop}

For $n \geq 2$ and $\lambda \in P$, let $B_n(\lambda)$ be the subcrystal of $B(\lambda)$ given by
\[ B_n(\lambda) = \{ \widetilde{f}_{i_r} \dots \widetilde{f}_{i_1} b_\lambda \,|\, r \in \mathbb{Z}_+, i_1, \dots, i_r \in \{ 0, \dots, n-1 \} \} \ \backslash\ \{ {\bf 0} \}, \]
i.e., $B_n(\lambda)$ is an extremal weight $\mathfrak{g}_n$-crystal generated by $b_\lambda$. 
In this case, the weight lattice associated with the subalgebra $\mathfrak{g}_n$ is obtained by 
the projection $\displaystyle P \ \twoheadrightarrow \ \bigoplus_{i=1}^n \mathbb{Z}\epsilon_i$ 
sending $\epsilon_i$ ($i \geq 1$) to $\epsilon_i$ if $i \leq n$ and to $0$ if $i > n$, and 
$\Lambda^\mathfrak{g}_0$ to $-\frac{\epsilon_\mathfrak{g}}{2} (\epsilon_1 + \dots + \epsilon_n)$.

Throughout this paper, we now assume that any $\lambda \in P$ has a nonnegative level. 
In fact, we can derive corresponding statements for $\lambda \in P$ with ${\rm lv}(\lambda) < 0$ from our results 
by considering dual crystals $B^\vee$ (cf. \cite[Section 7.4]{Kas95}).

\subsection{Tensor products of extremal weight crystals} \label{subsec:tensor prod of ext wt crystal}

For $\lambda = \sum \lambda_i \epsilon_i \in P$, set ${\rm supp}(\lambda) = \{\, i \in I \,|\, \lambda_i \neq 0 \,\}$.
Note that ${\rm supp}(\lambda)$ is a finite set if and only if ${\rm lv}(\lambda) = 0$. 
For $\lambda \in E$ with $r = |{\rm supp}(\lambda)|$, the $W$-orbit of $\lambda$ contains exactly one element 
of the form $\varpi_\mu$ for $\mu \in \mathcal{P}_r$. In this case, we denote $\mu = \lambda^\dagger$ and 
we know $B(\lambda) \cong B(\varpi_{\lambda^\dagger})$ by Proposition \ref{prop:W-orbit isomorphism}.

\begin{rem} \label{rem:Weyl group orbit D}
    The partition $\lambda^\dagger$ may not be well-defined in the weight lattice of $\mathfrak{g}_n$ for fixed $n \geq 2$.
    Indeed, for the Weyl group $W_n$ (the subgroup of $W$ generated by $s_0, s_1, \dots, s_{n-1}$) associated with $\mathfrak{g}_n$, 
    the $W_n$-orbit of $\lambda \in E$ may not contain any element of the form $\varpi_\mu$ for $\mu \in \mathcal{P}$. 
    However, the $W_N$-orbit (and so $W$-orbit) of $\lambda$  must contain such a (unique) element for sufficiently large $N (\geq |{\rm supp}(\lambda)|+2)$. 
    Thus, the weight $\varpi_{\lambda^\dagger}$ is well-defined in $P$.
\end{rem}

We explain the tensor product of two extremal weight crystals. For details, see \cite[Section 4]{NS12}.
\begin{prop} \label{prop:tensor product}
    The following statements hold.
    \begin{enumerate}
        \item When $\lambda \in E$ and $\mu \in E$, we have
            \begin{equation} \label{eqn:zero tensor zero}
                B(\lambda) \otimes B(\mu) \ \cong \ \bigoplus_{\nu \in E} B(\nu)^{\oplus {\tt LR}^{\nu^\dagger}_{\lambda^\dagger\,\mu^\dagger}},
            \end{equation}
            where ${\tt LR}^{\nu^\dagger}_{\lambda^\dagger\,\mu^\dagger}$ is the Littlewood-Richardson coefficient for partitions $\lambda^\dagger, \mu^\dagger$, and $\nu^\dagger$.

        \item When $\lambda \in P^+$ and $\mu \in P^+$, we have
            \begin{equation} \label{eqn:posi tensor posi}
                B(\lambda) \otimes B(\mu) \cong \bigoplus_{\substack{b \in B(\mu) \\ \lambda+{\rm wt}(b) \in P^+}} B(\lambda + {\rm wt}(b)).
            \end{equation}

        \item When $\lambda \in E$ and $\mu \in P^+$, suppose $\mu = \sum \mu_i \epsilon_i$ and 
            set $p$ to be the minimal integer such that $\mu_{p+1} = -\frac{\epsilon_\mathfrak{g}}{2} {\rm lv}(\mu)$. 
            In addition, when $\varpi_{\lambda^\dagger} = \sum_{i=1}^r m_i\epsilon_i$ with $r = |{\rm supp}(\lambda)|$, 
            let $\lambda_0 = \sum_{i=1}^r m_i \epsilon_{p+i} \in W\lambda \cap E$. Then we have
            \begin{equation} \label{eqn:zero tensor posi}
                B(\lambda) \otimes B(\mu) \,\cong\, B(\mu + \lambda_0).
            \end{equation}

        \item When $\lambda \in P^+$ and $\mu \in E$, we have
            \begin{equation} \label{eqn:posi tensor zero}
                B(\lambda) \otimes B(\mu) \,\cong\, \bigoplus_{\nu \in P^+_{[n]}(\lambda, \mu)} B(\nu)^{\oplus |B^{\max}_{n, \nu}|},
            \end{equation}
            where, for sufficiently large $n$, $P^+_{[n]}(\lambda, \mu)$ is the set of $\nu = \sum \nu_i\epsilon_i \in \lambda+\mu+Q$ such that 
                \[ 0 \geq \underbrace{\nu_1 \geq \dots \geq \nu_p}_{\nu_i > -\frac{\epsilon_\mathfrak{g}}{2} {\rm lv}(\lambda)} > 
                \underbrace{\nu_{p+1} = \dots = \nu_{p+q}}_{\nu_i = -\frac{\epsilon_\mathfrak{g}}{2} {\rm lv}(\lambda)} >
                \underbrace{\nu_{p+q+1} \geq \dots \geq \nu_n}_{\nu_i < -\frac{\epsilon_\mathfrak{g}}{2} {\rm lv}(\lambda)} < 
                \underbrace{\nu_{n+1} = \cdots}_{\nu_i = -\frac{\epsilon_\mathfrak{g}}{2} {\rm lv}(\lambda)} \]
            for some integers $p,q$ with $q \geq |\mu|$ (we replace $\nu_1$ with $|\nu_1|$ when $\mathfrak{g}_\infty = \mathfrak{d}_\infty$) and 
            \[ B^{\max}_{n, \nu} = \{\, b \in B_n(\lambda) \otimes B_n(\mu) \,|\, 
            {\rm wt}(b)= \nu,\ \widetilde{e}_ib = 0 \ (i = 0, \dots, n-1) \,\}. \]
    \end{enumerate}
\end{prop}
\begin{rem} \hfill

    (1) $B(\nu)$ appears in the decomposition of $B(\lambda) \otimes B(\mu)$ only if ${\rm lv}(\nu) = {\rm lv}(\lambda) + {\rm lv}(\mu)$. 

    (2) The authors of \cite{NS12} construct a bijection $\theta_n : P^+_{[n]}(\lambda, \mu) \to P^+_{[n+1]}(\lambda, \mu)$ for sufficiently large $n$ 
    such that $\nu \in P^+_{[n]}(\lambda, \mu)$ and $\theta_n(\nu) \in P^+_{[n+1]}(\lambda, \mu)$ are in the same $W$-orbit. 
    In addition, if $\theta_n(\nu) = w\nu$ for some $w \in W$, then there is a bijection $S_w : B^{\max}_{n+1, \nu} \to B^{\max}_{n, \nu}$. 
    Therefore, the decomposition \eqref{eqn:posi tensor zero} is independent of the choice of $n$.
\end{rem}

Now, we can explain that an extremal weight crystal is decomposed into the tensor product of two extremal weight crystals. 
Precisely speaking, the $W$-orbit of $\lambda \in P$ contains unique element $\nu = \sum \nu_i\epsilon_i$ such that
\begin{equation} \label{eqn:wt decomp}
    0 \geq \underbrace{\nu_1 \geq \dots \geq \nu_p}_{\nu_i > -\frac{\epsilon_\mathfrak{g}}{2} {\rm lv}(\lambda)} > 
    \underbrace{\nu_{p+1} \geq \dots \geq \nu_{p+q}}_{\nu_i < -\frac{\epsilon_\mathfrak{g}}{2} {\rm lv}(\lambda)} < 
    \underbrace{\nu_{p+q+1} = \cdots}_{\nu_i = -\frac{\epsilon_\mathfrak{g}}{2} {\rm lv}(\lambda)}
\end{equation}
for some $p, q \geq 0$ (with replacing $\nu_1$ with $|\nu_1|$ when $\mathfrak{g}_\infty = \mathfrak{d}_\infty$). 
For such $\nu$, set
\begin{equation} \label{eqn:wt decomp into zero and dominant}
    \nu^0 = \sum_{i=1}^q (\nu_{p+i}+\frac{\epsilon_\mathfrak{g}}{2} {\rm lv}(\lambda)) \epsilon_{p+i}, \quad 
    \nu^+ = \nu - \nu^0 \left( = \sum_{i=1}^p \nu_i\epsilon_i - \sum_{i=p+1}^\infty \frac{\epsilon_\mathfrak{g}}{2} {\rm lv}(\lambda)\epsilon_i \right).
\end{equation}
Then it is clear that $\nu^0 \in E$ and $\nu^+ \in P^+$.
\begin{prop} \label{prop:decomp into zero and posi}
    For $\lambda \in P$, there exist weights $\lambda^0 \in E$ and $\lambda^+ \in P^+$ such that
    \[ B(\lambda) \, \cong \, B(\lambda^0) \otimes B(\lambda^+). \]
    In addition, for $\lambda^0, \mu^0 \in E$ and $\lambda^+, \mu^+ \in P^+$, we have
    \[ B(\lambda^0) \otimes B(\lambda^+) \,\cong B(\mu^0)\, \otimes B(\mu^+) \iff \lambda^+ = \mu^+,\, \lambda^0 \in W\mu^0. \]
\end{prop}
\begin{proof}
    We use the same notation appearing in \eqref{eqn:wt decomp} and \eqref{eqn:wt decomp into zero and dominant}. 
    In this case, we know $B(\lambda) \cong B(\nu)$ by Proposition \ref{prop:W-orbit isomorphism}. 
    Moreover, the isomorphism \eqref{eqn:zero tensor posi} gives $B(\nu^0) \otimes B(\nu^+) \cong B(\nu)$.
    The last statement follows from \cite[Proposition 4.6]{NS12} (or one can easily check it using Proposition \ref{prop:W-orbit isomorphism}; 
    we leave it as an exercise).
\end{proof}

\begin{ex}
    For a simple description, write $\sum_i m_i \epsilon_i \in P$ as $(m_1, m_2, \dots)$. 
    Suppose $\mathfrak{g}_\infty = \mathfrak{c}_\infty$ and take $\lambda = (-2, -4, 1, -7, -5, 4, 0, -4, -4, \dots) \in P$. 
    Then the $W$-orbit of $\lambda$ contains unique $\nu = (0, -1, -2, -5, -7, -4, -4, \dots) \in W\lambda$ with
    \begin{eqnarray*}
        \nu^+ &=& (0, -1, -2, -4, -4, -4, \dots) \\
            &=& 4\Pi_0^\mathfrak{c} + 4\epsilon_1 + 3\epsilon_2 + 2\epsilon_3 = \Pi^\mathfrak{c}((3, 3, 2, 1), 4), \\
        \nu^0 &=& (0, 0, 0, -1, -3, 0, 0, 0, \dots) \\
            &=& -\epsilon_4-3\epsilon_5.
    \end{eqnarray*}
    Thus, we have $\lambda^0 = \varpi_{(3, 1)}$ and $\lambda^+ = \Pi^\mathfrak{c}((3, 3, 2, 1), 4)$.
\end{ex}

\section{Combinatorial realizations of extremal weight crystals} \label{sec:combinatorial realization}
\subsection{Spinor model} \label{subsec:spinor}
We recall a spinor model, which is introduced by Kwon \cite{K15, K16}, to describe the crystal base of integrable highest weight $\mathfrak{g}_\infty$-modules. 
\begin{rem}
    We focus a combinatorial description of $B(\lambda)$ for $\lambda \in P^+_{\rm int}$ (not $P^+$), 
    and we intentionally skip some related notions (see Remark \ref{rem:half-integer weight}).
    You can find the full description of a spinor model in \cite{K15,K16}; ${\bf T}^{\rm sp}$ in particular.
\end{rem}

For a semistandard tableau $T$ with two columns, denote by $T^{\tt L}$ (resp. $T^{\tt R}$) its left (resp. right) column tableau. 
For $T \in SST(\lambda(a, b, c))$, where $\lambda(a, b, c) = (2^{b+c}, 1^a)/(1^b)$ ($a, b, c \in \mathbb{Z}_+$) is a skew partition with two columns, 
suppose that the tableau $T'$ is obtained by sliding $T^{\tt R}$ by $k$ positions down for $0 \leq k \leq \min\{ a, b \}$. 
Define $\mathfrak{r}_T$ to be the maximal $k \geq 0$ such that $T'$ is semistandard.


For $a \in \mathbb{Z}_+$, let
\[ \mathbf{T}^\mathfrak{g}(a) \,=\, \{\, T \in SST(\lambda(a, b, c)) \,|\, (b, c) \in \mathcal{H}^\mathfrak{g},\ \mathfrak{r}_T \leq r^\mathfrak{g} \,\}, \]
where
\[
    \mathcal{H}^\mathfrak{g} = \begin{cases}
        \{ 0 \} \times \mathbb{Z}_+ & \mbox{if } \mathfrak{g} = \mathfrak{c} \\
        \mathbb{Z}_+ \times \mathbb{Z}_+ & \mbox{if } \mathfrak{g} = \mathfrak{b} \\
        2\mathbb{Z}_+ \times 2\mathbb{Z}_+ & \mbox{if } \mathfrak{g} = \mathfrak{d}
    \end{cases}, \quad
    r^\mathfrak{g} = \begin{cases}
        0 & \mbox{if } \mathfrak{g} = \mathfrak{b}\mbox{ or }\mathfrak{c}, \\
        1 & \mbox{if } \mathfrak{g} = \mathfrak{d}
    \end{cases},
\]
and
\[ \overline{\bf T}^\mathfrak{d}(0) = \bigsqcup_{(b, c) \in \mathcal{H}^\mathfrak{d}} SST(\lambda(0, b, c+1)). \]


For $(\lambda, \ell) \in \mathscr{P}(G)$, put $\ell(\lambda) = t$ and
\[
    \widehat{\bf T}^\mathfrak{g}(\lambda, \ell) = \begin{cases}
        {\bf T}^\mathfrak{g}(\lambda_\ell) \times \dots \times {\bf T}^\mathfrak{g}(\lambda_1) & \mbox{if } t \leq \ell, \\
        \overline{\bf T}^\mathfrak{d}(0)^{t-\ell} \times {\bf T}^\mathfrak{d}(\lambda_{2\ell-t}) \times \dots \times {\bf T}^\mathfrak{d}(\lambda_1) & \mbox{if } t > \ell.
    \end{cases}
\]
\begin{df}
    A spinor model ${\bf T}^\mathfrak{g}(\lambda, \ell)$ of shape $(\lambda, \ell) \in \mathscr{P}(G)$ 
    is the set of $(T_\ell, \dots, T_1) \in \widehat{\bf T}^\mathfrak{g}(\lambda, \ell)$ such that the pair $(T_{i+1}, T_i)$ is {\it admissible} for all $i = 1, \dots, \ell-1$ 
    (see \cite[Definition 6.7]{K15} and \cite[Definition 3.4]{K16}).
\end{df}
\begin{rem}
    The admissibility condition given in \cite{K16} seems different from that in \cite{K15}. 
    But they adopt different conventions and they are essentially the same 
    in terms that ${\bf T}^\mathfrak{g}(\lambda, \ell)$ is isomorphic to the crystal base of an integrable highest weight $\mathfrak{g}_\infty$-module (see Theorem \ref{thm:spinor}).
    In this paper, we follow the convention of \cite{K15}.
\end{rem}

\begin{thm}[{\cite[Theorem 7.4]{K15}}, {\cite[Theorem 4.4]{K16}}] \label{thm:spinor}
    For $(\lambda, \ell) \in \mathscr{P}(G)$, the set ${\bf T}^\mathfrak{g}(\lambda, \ell)$ is a $\mathfrak{g}_\infty$-crystal and 
    is isomorphic to $B(\Pi^\mathfrak{g}(\lambda, \ell))$ as $\mathfrak{g}_\infty$-crystals. 
\end{thm}
The character of ${\bf T}^\mathfrak{g}(\lambda, \ell)$ is defined to be
\[ {\rm ch}\,{\bf T}^\mathfrak{g}(\lambda, \ell) \,=\, t^\ell \sum_{(T_\ell, \dots, T_1) \in {\bf T}^\mathfrak{g}(\lambda, \ell)} \prod_{i=1}^\ell \ {\bf x}^{T_i} \]
where $t$ is a formal symbol and ${\bf x}^T = \prod_{i=1}^\infty x_i^{m_i}$, 
with $m_i$ being the number of appearances of $i \geq 1$ in a semistandard tableau $T$. 
Indeed, we understand $x_i = e^{\epsilon_i}$ and $t = e^{\Pi^\mathfrak{g}_0}$ when we consider them as elements in the group algebra $\mathbb{Z}[P]$. 
An explicit formula of the character of a spinor model will be explained in Section \ref{sec:Jacobi-Trudi}.

For $n \in \mathbb{N}$ and $(\lambda, \ell) \in \mathscr{P}(G)$, define ${\bf T}_n^\mathfrak{g}(\lambda, \ell)$ 
to be the set of $T \in {\bf T}^\mathfrak{g}(\lambda, \ell)$ such that all entries of $T$ are filled with entries in $\{ 1, \dots, n \}$. 
By Theorem \ref{thm:spinor}, there is an isomorphism of $\mathfrak{g}_n$-crystals
\begin{equation} \label{eqn:iso between spinor and crystal}
    {\bf T}_n^\mathfrak{g}(\lambda, \ell) \cong B_n(\Pi^\mathfrak{g}(\lambda, \ell)).
\end{equation}

\subsection{Kashiwara-Nakashima tableaux} \label{subsec:KN}
Before we consider KN tableaux, we introduce the notion of generalized partitions and semistandard tableaux of a generalized partition shape. 
A generalized partition is a sequence $\lambda = (\lambda_1, \lambda_2, \dots, \lambda_n)$ of integers such that $\lambda_1 \geq \lambda_2 \geq \dots \geq \lambda_{n-1} \geq |\lambda_n|$. 
When $\lambda_n < 0$, the Young diagram for a generalized partition $\lambda$ is the Young diagram for $(\lambda_1, \dots, \lambda_{n-1}, |\lambda_n|)$ 
with coloring at the $n$-th row.
As similarly as usual partitions, set the length $\ell(\lambda)$ of a generalized partition $\lambda$ to be the number of its nonzero components 
and $\varpi_\lambda = \sum_{i=1}^{\ell(\lambda)} \lambda_i \epsilon_i$. 
For $n \in \mathbb{Z}_+$, we denote by $\mathcal{GP}_n$ the set of generalized partitions $\lambda$ with $\ell(\lambda) \leq n$, 
and set $\mathcal{GP}_n^- = \mathcal{GP}_n \,\backslash\, \mathcal{P}_n$ to be the set of generalized partitions $\lambda = (\lambda_1, \dots, \lambda_n)$ with $\lambda_n < 0$. 
A semistandard tableau of shape $\lambda = (\lambda_1, \dots, \lambda_n) \in \mathcal{GP}_n^-$ 
is a semistandard tableau of shape $(\lambda_1, \dots, \lambda_{n-1}, |\lambda_n|) \in \mathcal{P}_n$ with the coloring at its $n$-th row. 
For example, let $T_\lambda$ be the semistandard tableau of shape $\lambda = (\lambda_1, \dots, \lambda_n) \in \mathcal{GP}_n^-$ such that 
\begin{equation} \label{eqn:ext wt vec for gp}
    T_\lambda(i, j) = \begin{cases}
        \overline{n} & \mbox{if } i = 1,\ 1 \leq j \leq |\lambda_n|, \\
        i-1 & \mbox{if } 2 \leq i \leq \ell(\lambda),\ 1 \leq j \leq |\lambda_n|, \\
        i & \mbox{if } 1 \leq i \leq \ell(\lambda),\ j > |\lambda_n|.
    \end{cases}
\end{equation}
In particular, if we take $n=4$ and $\lambda = (4, 3, 1, -1) \in \mathcal{GP}^-_n$, then we have the following $T_\lambda$.
\[ T_\lambda = \begin{ytableau} \overline{4} & 1 & 1 & 1 \\ 1 & 2 & 2 \\ 2 \\ *(black) {\color{white} 3} \end{ytableau} \]

For a positive integer $n$, let $\mathcal{I}_n^\mathfrak{g}$ be the following ordered sets.
\begin{eqnarray*}
    \mathcal{I}_n^\mathfrak{b} &=& \{\, \overline{n} < \dots < \overline{1} < 0 < 1 < \dots < n \,\} \\
    \mathcal{I}_n^\mathfrak{c} &=& \{\, \overline{n} < \dots < \overline{1} < 1 < \dots < n \,\} \\
    \mathcal{I}_n^\mathfrak{d} &=& \left\{\, \overline{n} < \dots < \overline{2} < \begin{matrix} \overline{1} \\ 1 \end{matrix} < 2 < \dots < n \,\right\} \quad (n \geq 2)
\end{eqnarray*}
Here, ($1$, $\overline{1}$) is the unique non-comparable pair in $\mathcal{I}_n^\mathfrak{d}$. 
An $\mathcal{I}_n^\mathfrak{g}$-tableau is said to be $\mathcal{I}_n^\mathfrak{g}$-semistandard if entries in each row are weakly increasing from left to right and 
entries in each column are strictly increasing (not weakly decreasing when $\mathfrak{g} = \mathfrak{d}$) from top to bottom with the following exceptions:
\begin{itemize}
    \item when $\mathfrak{g} = \mathfrak{b}$, the letter $0$ cannot be repeated in each row and can be repeated in each column.
    \item when $\mathfrak{g} = \mathfrak{d}$, two letters $1$ and $\bar{1}$ can appear alternatively and successively in each column.
\end{itemize}
By definition, two letters $1$ and $\bar{1}$ cannot appear simultaneously in each row of an $\mathcal{I}^\mathfrak{d}_n$-semistandard tableau.
\begin{df}
    Suppose that $C$ is an $\mathcal{I}_n^\mathfrak{g}$-column tableau.
    For a positive integer $z$, let
    \[ N_C(z;n) = \left| \{ x \in C \,|\, x \leq \overline{z} \mbox{ or } x \geq z \} \right|. \]
    We say that a column tableau $C$ is $n$-admissible if ${\rm ht}(C) \leq n$ and $N_C(z;n) \leq n-z+1$ for all $1 \leq z \leq n$. 
    Here, the number $N_C(1; n)$ captures the multiplicity of $1$ and $\bar{1}$ in $C$ when $\mathfrak{g} = \mathfrak{d}$.
\end{df}

The condition $N_C(z;n) \leq n-z+1$ is equivalent to the following condition: 
if $C(p, 1) = \overline{z}$ and $C(q, 1) = z$ for some $1 \leq z \leq n$ and $1 \leq p < q \leq {\rm ht}(C)$, then we have $(q-p)+(n-z+1) > {\rm ht}(C)$. 
When $\mathfrak{g} = \mathfrak{d}$, there might be $p > q$ such that $C(p, 1) = \overline{1}$ and $C(q, 1) = 1$. 
In this case, the condition is replaced with $(p-q) + n > {\rm ht}(C)$.
\begin{rem} \label{rem:admissible}
    Any $\mathcal{I}_n^\mathfrak{g}$-column tableau is always $N$-admissible (as an $\mathcal{I}^\mathfrak{g}_n$-tableau) 
    for sufficiently large $N \,(\geq {\rm ht}(C))$ regardless of whether it is $n$-admissible or not.
\end{rem}

\begin{df}[\cite{KN94}] \label{def:KN tableau}
    A ($\mathfrak{g}_n$-type) KN tableau of shape $\lambda \in \mathcal{P}_n$ ($\lambda \in \mathcal{GP}_n$ when $\mathfrak{g} = \mathfrak{d}$) is 
    an $\mathcal{I}_n^\mathfrak{g}$-semistandard tableau $T$ of shape $\lambda$ such that 
    all columns are $n$-admissible and it satisfies the following conditions:
    \begin{itemize}
        \item when $\mathfrak{g} = \mathfrak{c}$
        \begin{enumerate}
            \item[($\mathfrak{c}$-1)] if either $T(p, j) = \overline{a}, T(q, j) = \overline{b}, T(r, j) = b, T(s, j+1) = a$ or 
            $T(p, j) = \overline{a}, T(q, j+1) = \overline{b}, T(r, j+1) = b, T(r, j+1) = a$ for some $1 \leq b \leq a \leq n$ and $p \leq q < r \leq s$, 
            then we have $(q-p)+(s-r) < a-b$.
        \end{enumerate}
        \item when $\mathfrak{g} = \mathfrak{b}$
        \begin{enumerate}
            \item[($\mathfrak{b}$-1)] if either $T(p, j) = \overline{a}, T(q, j) = \overline{b}, T(r, j) = b, T(s, j+1) = a$ or 
            $T(p, j) = \overline{a}, T(q, j+1) = \overline{b}, T(r, j+1) = b, T(r, j+1) = a$ for some $1 < b \leq a \leq n$ and $p \leq q < r \leq s$, 
            then we have $(q-p)+(s-r) < a-b$.
            \item[($\mathfrak{b}$-2)] suppose $T(p, j) = \overline{a}, T(s, j+1) = a$ for some $j$, $1 < a \leq n$ and $p < s$. 
            If either $T(q, j), T(q+1, j) \in \{ 1, 0, \bar{1} \}$ or $T(q, j+1), T(q+1, j+1) \in \{ 1, 0, \bar{1} \}$ for some $p \leq q < r = q+1 \leq s$, 
            then we have $(q-p)+(s-r) < a-1$.
            \item[($\mathfrak{b}$-3)] there is no $p < q$ such that $T(p, j) \in \{ \bar{1}, 0 \}, T(q, j+1) \in \{ 0, 1 \}$ for some $j$.
        \end{enumerate}
        \item when $\mathfrak{g} = \mathfrak{d}$
        \begin{enumerate}
            \item[($\mathfrak{d}$-1)] suppose $\ell(\lambda) = n$ with $\lambda_n > 0$. 
            For each $j \geq 1$ such that the height of the $j$-th column from the left is $n$, 
            if $\overline{1}$ (resp. $1$) fills the $(k, j)$-th entry of $T$, then $n-k$ is even (resp. odd).
            \item[($\mathfrak{d}$-2)] suppose $\ell(\lambda) = n$ with $\lambda_n < 0$. 
            For each $j \geq 1$ such that the height of the $j$-th column from the left is $n$, 
            if $\overline{1}$ (resp. $1$) fills the $(k, j)$-th entry of $T$, then $n-k$ is odd (resp. even).
            \item[($\mathfrak{d}$-3)] if either $T(p, j) = \overline{a}, T(q, j) = \overline{b}, T(r, j) = b, T(s, j+1) = a$ or 
            $T(p, j) = \overline{a}, T(q, j+1) = \overline{b}, T(r, j+1) = b, T(r, j+1) = a$ for some $1 < b \leq a \leq n$ and $p \leq q < r \leq s$, 
            then we have $(q-p)+(s-r) < a-b$.
            \item[($\mathfrak{d}$-4)] suppose if $T(p, j) = \overline{a}, T(s, j+1) = a$ for some $j$, $1 < a \leq n$ and $p < s$. 
            If either $T(q, j), T(q+1, j) \in \{ 1, \bar{1} \}$ with $T(q, j) \neq T(r, j)$ or $T(q, j+1), T(r, j+1) \in \{ 1, \bar{1} \}$ with $T(q, j+1) \neq T(r, j+1)$ for some $p \leq q < r=q+1 \leq s$, 
            then we have $(q-p)+(s-r) < a-1$.
            \item[($\mathfrak{d}$-5)] there is no $p < q$ such that $T(p, j) \in \{ \bar{1}, 1 \}, T(q, j+1) \in \{ \bar{1}, 1 \}$ for some $j$.
            \item[($\mathfrak{d}$-6)] suppose that $T(p, j) = \overline{a}, T(s, j+1) = a$ for some $j$, $1 < a \leq n$, and $p<s$. 
            If $T(q, j+1) \in \{ \bar{1}, 1 \}$ and $T(r, j) \in \{ \bar{1}, 1 \}$ for some $p \leq q < r \leq s$ and
            $s-q+1$ is even when $T(q, j+1) = T(r, j)$ and $s-q+1$ is odd when $T(q, j+1) \neq T(r, j)$, then $s-p < a-1$.
        \end{enumerate}
    \end{itemize}

    Denote by ${\bf KN}_n^\mathfrak{g}(\lambda)$ the set of ($\mathfrak{g}_n$-type) KN tableau of shape $\lambda$.
\end{df}
For example, ${\bf KN}_n^\mathfrak{g}((1^h))$ is the set of $n$-admissible column tableaux of height $h \geq 0$ except when $\mathfrak{g} = \mathfrak{d}$ with $n = h$. 
Note that the set of $\mathfrak{d}_n$-type $n$-admissible column tableaux of height $n$ is ${\bf KN}_n^\mathfrak{d}((1^n)) \cup {\bf KN}_n^\mathfrak{d}((1^{n-1}, -1))$.

\begin{rem}
    (1) In this article, we take simple roots as the negative of simple roots taken in \cite{KN94}. 
    Due to this choice, the condition for KN tableaux seems different from that given in \cite{KN94}, but they are equivalent.

    (2) There is an equivalent condition for a tableau to be a KN tableau, which is given in \cite{Le02, Le03}; its split form is semistandard (with additional conditions when $\mathfrak{g} = \mathfrak{d}$).
\end{rem}

For $\lambda \in \mathcal{P}_n$ ($\lambda \in \mathcal{GP}_n$ when $\mathfrak{g} = \mathfrak{d}$), 
it is proved in \cite{KN94} that ${\bf KN}_n^\mathfrak{g}(\lambda)$ is a $\mathfrak{g}_n$-crystal and there is an isomorphism of $\mathfrak{g}_n$-crystals
\begin{equation} \label{eqn:iso between KN and crystal}
    {\bf KN}_n^\mathfrak{g}(\lambda) \, \cong \, B_n(\varpi_\lambda).
\end{equation}
For example, the crystal graph of ${\bf KN}_n^\mathfrak{g}((1))$ is as follows.
\[ {\bf KN}^\mathfrak{b}_n((1)) \quad:\quad \raisebox{-1em}{\begin{tikzpicture}
    \node at (-1.4, 0) {$\begin{ytableau} \overline{1} \end{ytableau}$};
    \node at (-2.8, 0) {$\cdots$};
    \node at (-4.2, 0) {$\begin{ytableau} \overline{n} \end{ytableau}$};
    \node at (0, 0) {$\begin{ytableau} 0 \end{ytableau}$};
    \node at (1.4, 0) {$\begin{ytableau} 1 \end{ytableau}$};
    \node at (2.8, 0) {$\cdots$};
    \node at (4.2, 0) {$\begin{ytableau} n \end{ytableau}$};

    \draw[->] (-0.3-1.4*0.5, 0) -- (0.3-1.4*0.5, 0);
    \node [above] at (-0.7, 0) {\scriptsize $0$};
    \draw[->] (-0.3+1.4*0.5, 0) -- (0.3+1.4*0.5, 0);
    \node [above] at (0.7, 0) {\scriptsize $0$};
    \draw[->] (-0.3+1.4*1.5, 0) -- (0.3+1.4*1.5, 0);
    \node [above] at (2.1, 0) {\scriptsize $1$};
    \draw[->] (-0.3+1.4*2.5, 0) -- (0.3+1.4*2.5, 0);
    \node [above] at (3.5, 0) {\scriptsize $n-1$};
    \draw[->] (-0.3-1.4*1.5, 0) -- (0.3-1.4*1.5, 0);
    \node [above] at (-2.1, 0) {\scriptsize $1$};
    \draw[->] (-0.3-1.4*2.5, 0) -- (0.3-1.4*2.5, 0);
    \node [above] at (-3.5, 0) {\scriptsize $n-1$};
\end{tikzpicture}} \]
\[ {\bf KN}^\mathfrak{c}_n((1)) \quad:\quad \raisebox{-1em}{\begin{tikzpicture}
    \node at (-0.7, 0) {$\begin{ytableau} \overline{1} \end{ytableau}$};
    \node at (-2.1, 0) {$\cdots$};
    \node at (-3.5, 0) {$\begin{ytableau} \overline{n} \end{ytableau}$};
    \node at (0.7, 0) {$\begin{ytableau} 1 \end{ytableau}$};
    \node at (2.1, 0) {$\cdots$};
    \node at (3.5, 0) {$\begin{ytableau} n \end{ytableau}$};

    \draw[->] (-0.3-1.4*0, 0) -- (0.3-1.4*0, 0);
    \node [above] at (0, 0) {\scriptsize $0$};
    \draw[->] (-0.3+1.4*1, 0) -- (0.3+1.4*1, 0);
    \node [above] at (1.4, 0) {\scriptsize $1$};
    \draw[->] (-0.3+1.4*2, 0) -- (0.3+1.4*2, 0);
    \node [above] at (2.8, 0) {\scriptsize $n-1$};
    \draw[->] (-0.3-1.4*1, 0) -- (0.3-1.4*1, 0);
    \node [above] at (-1.4, 0) {\scriptsize $1$};
    \draw[->] (-0.3-1.4*2, 0) -- (0.3-1.4*2, 0);
    \node [above] at (-2.8, 0) {\scriptsize $n-1$};
\end{tikzpicture}} \]
\[ {\bf KN}^\mathfrak{d}_n((1)) \quad:\quad \raisebox{-2em}{\begin{tikzpicture}
    \node at (0, 0.5) {$\begin{ytableau} \overline{1} \end{ytableau}$};
    \node at (0, -0.5) {$\begin{ytableau} 1 \end{ytableau}$};
    \node at (-1.4, 0) {$\begin{ytableau} \overline{2} \end{ytableau}$};
    \node at (-2.8, 0) {$\cdots$};
    \node at (-4.2, 0) {$\begin{ytableau} \overline{n} \end{ytableau}$};
    \node at (1.4, 0) {$\begin{ytableau} 2 \end{ytableau}$};
    \node at (2.8, 0) {$\cdots$};
    \node at (4.2, 0) {$\begin{ytableau} n \end{ytableau}$};

    \draw[->] (-0.3-1.4*0.5, 0.05) -- (0.3-1.4*0.5, 0.5);
    \node [above] at (-0.7, 0.3) {\scriptsize $1$};
    \draw[->] (-0.3-1.4*0.5, -0.05) -- (0.3-1.4*0.5, -0.5);
    \node [below] at (-0.7, -0.3) {\scriptsize $0$};
    \draw[->] (-0.3+1.4*0.5, 0.5) -- (0.3+1.4*0.5, 0.05);
    \node [above] at (0.7, 0.3) {\scriptsize $0$};
    \draw[->] (-0.3+1.4*0.5, -0.5) -- (0.3+1.4*0.5, -0.05);
    \node [below] at (0.7, -0.3) {\scriptsize $1$};
    \draw[->] (-0.3+1.4*1.5, 0) -- (0.3+1.4*1.5, 0);
    \node [above] at (2.1, 0) {\scriptsize $2$};
    \draw[->] (-0.3+1.4*2.5, 0) -- (0.3+1.4*2.5, 0);
    \node [above] at (3.5, 0) {\scriptsize $n-1$};
    \draw[->] (-0.3-1.4*1.5, 0) -- (0.3-1.4*1.5, 0);
    \node [above] at (-2.1, 0) {\scriptsize $2$};
    \draw[->] (-0.3-1.4*2.5, 0) -- (0.3-1.4*2.5, 0);
    \node [above] at (-3.5, 0) {\scriptsize $n-1$};
\end{tikzpicture}} \]
We have ${\rm wt}(T_\lambda) = \varpi_\lambda$ (cf. \eqref{eqn:ext wt vec for partition} and \eqref{eqn:ext wt vec for gp}) 
and $T_\lambda \in {\bf KN}_n^\mathfrak{g}(\lambda)$ corresponds to $b_{\varpi_\lambda} \in B_n(\varpi_\lambda)$ under the isomorphism \eqref{eqn:iso between KN and crystal} 
since $T_\lambda$ (resp. $b_{\varpi_\lambda}$) is the unique element in ${\bf KN}^\mathfrak{g}_n(\lambda)$ (resp. $B_n(\varpi_{\lambda})$) of weight $\varpi_\lambda$. 
Thus, ${\bf KN}^\mathfrak{g}_n(\lambda)$ is an extremal weight $\mathfrak{g}_n$-crystal generated by $T_\lambda$.

For $(\lambda, \ell) \in \mathscr{P}(G)$ with $\lambda_1 \leq n$ and $\ell(\lambda) \leq \ell$, define a partition
\[ \rho_n(\lambda, \ell) = (n-\lambda_\ell, n-\lambda_{\ell-1}, \dots, n-\lambda_1)', \]
and for $(\lambda, \ell) \in \mathscr{P}(G)$ with $\lambda_1 \leq n$ and $\ell(\lambda) = t > \ell$, define a generalized partition
\[ \rho_n(\lambda, \ell) = (\rho_1, \dots, \rho_{n-1}, \ell-t), \]
where $(\rho_1, \dots, \rho_{n-1}) = \rho_n((\lambda_1, \dots, \lambda_{2\ell-t}, 1^{t-\ell}), \ell)$ (cf. \eqref{eqn:well-defined partition}).
For example, when $\mathfrak{g} = \mathfrak{d}$ and $n=4$, we have
\[ \rho_n((3,2,1,1,1), 4) \,=\, \ydiagram{4,3,1}*[*(black)]{0,0,0,1}. \]

We already have an isomorphism of $\mathfrak{g}_n$-crystals between ${\bf T}^\mathfrak{g}_n(\lambda, \ell)$ and ${\bf KN}^\mathfrak{g}_n(\rho_n(\lambda, \ell))$ 
(cf. \eqref{eqn:iso between spinor and crystal} and \eqref{eqn:iso between KN and crystal}), 
but we can find an explicit description of the isomorphism in \cite{K18a} when $\mathfrak{g} = \mathfrak{b}$ or $\mathfrak{c}$, 
and in \cite{JangKwon21} when $\mathfrak{g} = \mathfrak{d}$.
\begin{prop} \label{prop:bij spinor to KN}
    For $(\lambda, \ell) \in \mathscr{P}(G)$, we have an isomorphism of $\mathfrak{g}_n$-crystals.
    \begin{equation} \label{eqn:bij spinor to KN}
        \xymatrixcolsep{2pc}\xymatrixrowsep{0.5pc}\xymatrix{
            (\,\cdot\,)^{\tt ad} \,:\, {\bf T}_n^\mathfrak{g}(\lambda,\ell)  \ar@{->}[r]  & \ {\bf KN}_n^\mathfrak{g}(\rho_n(\lambda,\ell)) \\
            {\bf T}=(T_\ell,\dots,T_1)  \ar@{|->}[r] & {\bf T}^{\tt ad}:= (T_\ell^{\tt ad},\dots, T_1^{\tt ad}) }
    \end{equation}
    Here, the $i$-th column of ${\bf T}^{\tt ad}$ from the left is $T_i^{\tt ad}$ for $1 \leq i \leq \ell$.
\end{prop}
\begin{rem}
    The $\epsilon$ appearing in \cite[(3.10)]{JangKwon21} should be corrected as follows:
    \[ \Phi_a(T) = \mathcal{F}^{n-a-\epsilon}(T_-, \widetilde{T}_+), \]
    where
    \[ \epsilon = \begin{cases}
        1 & \mbox{if $a \neq n$ and ${\rm ht}(T_+) - a$ is odd,} \\
        0 & \mbox{otherwise.}
    \end{cases} \]
    Note that $\Phi_a(T)$ has residue $\epsilon$.
\end{rem}

\subsection{KN tableaux of type $\mathfrak{g}_\infty$} \label{subsec:KNL}
For any $\lambda \in \mathcal{P}$ and a sufficiently large $n$, we easily observe that ${\bf KN}_n^\mathfrak{g}(\lambda) \subseteq {\bf KN}_{n+1}^\mathfrak{g}(\lambda)$. 
Now we introduce a tableau model introduced by Lecouvey \cite{Le09}.
\begin{df}
    For $\lambda \in \mathcal{P}$, define
    \[ {\bf KN}^\mathfrak{g}(\lambda) = \bigcup_{n \geq \ell(\lambda)} {\bf KN}_n^\mathfrak{g}(\lambda), \]
    where the union is over $n > \ell(\lambda)$ when $\mathfrak{g} = \mathfrak{d}$, 
    and call it the set of $\mathfrak{g}_\infty$-type KN tableaux of shape $\lambda$.
    It is the set of $\mathcal{I}^\mathfrak{g}$-semistandard tableaux of shape $\lambda$ satisfying the conditions appearing in Definition \ref{def:KN tableau}, 
    where $\mathcal{I}^\mathfrak{g}$ are the following (partially) ordered sets.
    \begin{eqnarray*}
        \mathcal{I}^\mathfrak{b} &=& \{\, \dots < \overline{n} < \dots < \overline{1} < 0 < 1 < \dots < n < \dots \,\} \\
        \mathcal{I}^\mathfrak{c} &=& \{\, \dots < \overline{n} < \dots < \overline{1} < 1 < \dots < n < \dots \,\} \\
        \mathcal{I}^\mathfrak{d} &=& \left\{\, \dots < \overline{n} < \dots < \overline{2} < \begin{matrix} \overline{1} \\ 1 \end{matrix} < 2 < \dots < n < \dots \,\right\}
    \end{eqnarray*}
    Here, two letters $1$ and $\bar{1}$ in $\mathcal{I}^\mathfrak{d}$ are not comparable.
\end{df}
We remark that we can always find $\lambda^\dagger \in \mathcal{P}$ for $\lambda \in E$ 
such that $B(\lambda) \cong B(\varpi_{\lambda^\dagger})$ as $\mathfrak{g}_\infty$-crystals (see Remark \ref{rem:Weyl group orbit D}), 
and hence it suffices to define ${\bf KN}^\mathfrak{g}(\lambda)$ for $\lambda \in \mathcal{P}$ (not $\mathcal{GP}$). 
For example, ${\bf KN}^\mathfrak{g}((1^h))$ is the set of $\mathcal{I}^\mathfrak{g}$-column tableaux of height $h \geq 0$ 
since any $\mathcal{I}^\mathfrak{g}$-column tableau is $N$-admissible for sufficiently large $N$ by Remark \ref{rem:admissible}. 
Indeed, the $n$-admissible column condition turns out to be redundant when considering ${\bf KN}^\mathfrak{g}(\lambda)$.

We define a $\mathfrak{g}_\infty$-crystal structure of ${\bf KN}^\mathfrak{g}(\lambda)$ as the induced one 
from the $\mathfrak{g}_n$-crystal structure of ${\bf KN}^\mathfrak{g}_n(\lambda)$ (cf. \cite{Le09}), which is constructed as follows:
by definition, a tableau $T \in {\bf KN}^\mathfrak{g}(\lambda)$ is contained in ${\bf KN}^\mathfrak{g}_n(\lambda)$ for some $n \in \mathbb{N}$. 
For $k < n$, set $\widetilde{e}_k T$ to be the tableau considering $T$ as an element of ${\bf KN}^\mathfrak{g}_n(\lambda)$. 
For $k \geq n$, we already know $T \in {\bf KN}^\mathfrak{g}_n \subseteq {\bf KN}^\mathfrak{g}_{k+1}(\lambda)$ and set $\widetilde{e}_k T$ to be the tableau 
considering $T$ as an element of ${\bf KN}^\mathfrak{g}_{k+1}(\lambda)$. 
The other functions $\widetilde{f}_i$ and $\varepsilon_i, \varphi_i$ are defined similarly. 

We know that $B_n(\varpi_\mu) \cong {\bf KN}^\mathfrak{g}_n(\mu)$ is a $\mathfrak{g}_n$-subcrystal 
of $B_{n+1}(\varpi_\mu) \cong {\bf KN}^\mathfrak{g}_{n+1}(\mu)$ generated by $b_{\varpi_\mu}$ for $\mu \in \mathcal{P}$ and sufficiently large $n$, 
which guarantees the well-definedness of our $\mathfrak{g}_\infty$-crystal.
Note that the weight of $T$ does not depend on $n$ and it is clearly well-defined. 

By construction, we easily show that ${\bf KN}^\mathfrak{g}(\lambda)$ is connected. 
In particular, any $T \in {\bf KN}^\mathfrak{g}(\lambda)$ is connected to $T_\lambda$.
Now, we give the first main result of this paper.
\begin{thm} \label{thm:level zero realization}
    For $\lambda \in \mathcal{P}$, a $\mathfrak{g}_\infty$-crystal ${\bf KN}^\mathfrak{g}(\lambda)$ is isomorphic to the extremal weight crystal $B(\varpi_\lambda)$ as $\mathfrak{g}_\infty$-crystals. 
    In particular, $T_\lambda \in {\bf KN}^\mathfrak{g}(\lambda)$ corresponds to $b_{\varpi_\lambda} \in B(\varpi_\lambda)$ under the isomorphism.
\end{thm}
\begin{proof}
    Let $\Lambda^+ = \Pi^\mathfrak{g}(\lambda', \lambda_1) \in P^+$ and $\Lambda^- = \lambda_1\Pi_0^\mathfrak{g} \in P^+$. Note that $\varpi_\lambda = \Lambda^+ - \Lambda^-$. 
    Take extremal weight vectors $b^+ = b_{\Lambda^+} \in B(\Lambda^+)$ and $b^- = b_{-\Lambda^-} \in B(-\Lambda^-)$. 
    Fix an integer $n \geq \ell = \ell(\lambda)$. By Proposition \ref{prop:bij spinor to KN}, 
    $b^+$ is $\mathfrak{g}_n$-crystal equivalent to $T^+ \in {\bf KN}_n^\mathfrak{g}(\rho_n(\lambda', \lambda_1))$ whose $i$th row is filled with $\overline{n-i+1}$ for $1 \leq i \leq \lambda'_{\lambda_1}$. 
    Similarly, $b^-$ is $\mathfrak{g}_n$-crystal equivalent to $T^- \in {\bf KN}_n^\mathfrak{g}((\lambda_1^n))$ whose $i$th row is filled with $i$ for $1 \leq i \leq n$.

    For $\mathfrak{g}_n$-type KN tableaux $T$ and $S$, we know that $T \otimes S$ is $\mathfrak{g}_n$-crystal equivalent to the tableau $(T \to S)$, which is the insertion of $T$ into $S$ (cf. \cite{Le02, Le03}). 
    Then we can directly check that $(T^+ \to T^-) = T_\lambda$, which implies 
    that the connected component of $T_\lambda$ is isomorphic to the connected component of $T^+ \otimes T^-$ in ${\bf KN}_n^\mathfrak{g}(\rho_n(\lambda', \lambda_1)) \otimes {\bf KN}_n^\mathfrak{g}((\lambda_1^n))$ as $\mathfrak{g}_n$-crystals. 
    Since ${\bf KN}^\mathfrak{g}_n(\lambda)$ is connected, we can find an embedding of $\mathfrak{g}_n$-crystals
    \[ {\bf KN}_n^\mathfrak{g}(\lambda) \hookrightarrow B_n^\mathfrak{g}(\Lambda^+) \otimes B_n^\mathfrak{g}(-\Lambda^-), \]
    which sends $T_\lambda$ to $b^+ \otimes b^-$. 
    By construction, we can extend the embeddings of $\mathfrak{g}_n$-crystals to an embedding of $\mathfrak{g}_\infty$-crystals
    \[ {\bf KN}^\mathfrak{g}(\lambda) \hookrightarrow B(\Lambda^+) \otimes B(-\Lambda^-), \]
    which sends $T_\lambda$ into $b^+ \otimes b^-$. 
    In addition, since ${\bf KN}^\mathfrak{g}(\lambda)$ is also connected, $C(T_\lambda) = {\bf KN}^\mathfrak{g}(\lambda)$ is isomorphic to $C(b^+ \otimes b^-)$. 
    By Lemma \ref{lem:summand in tensor}, $C(b^+ \otimes b^-)$ is isomorphic to $C(b_{\varpi_\lambda})$. 
    By Proposition \ref{prop:conn}, $C(b_{\varpi_\lambda}) = B(\varpi_\lambda)$ and it's done. 
    The last statement is clear from the above argument.
\end{proof}

\begin{rem}
    By Theorem \ref{thm:level zero realization}, for $\lambda \in \mathcal{P}$, 
    $B(\varpi_\lambda) \cong {\bf KN}^\mathfrak{g}(\lambda)$ has infinitely many elements of the same weight. 
    Thus, we cannot define the character of $B(\lambda)$ for $\lambda \in E$ in general.
\end{rem}

\section{Jacobi-Trudi type character formula} \label{sec:Jacobi-Trudi}
\subsection{$E$-expansion of a spinor model} \label{subsec:E-expansion}
Let ${\bf x} = \{ x_1, x_2, \dots \}$ and ${\bf x}_n = \{ x_1, \dots, x_n \}$ for $n \geq 1$ be sets of variables. 
To indicate variables of given symmetric functions or polynomials, we write the letter set. 
For example, the Schur function in variables ${\bf x}$ (resp. ${\bf x}_n$) is denoted by $s_\lambda({\bf x})$ (resp. $s_\lambda({\bf x}_n)$).
In addition, for a subset (not necessarily a subcrystal) $X$ of a crystal $B$, set ${\rm ch}\,X$ to be the weight generating function for $X$ 
(with an abuse of terminology).

\begin{lem} \label{lem:spinor residue character}
    For $a, b, c \in \mathbb{Z}_+$ and $0 \leq k \leq \min\{ a, b \}$, let $SST(\lambda(a, b, c))_{(k)}$ be the set of $T \in SST(\lambda(a, b, c))$ with $\mathfrak{r}_T = k$. Then
    \[ {\rm ch}\, SST(\lambda(a, b, c))_{(k)} = s_{(a+b+c-k, c+k)'}({\bf x}). \]
\end{lem}
\begin{proof}
    It is clear that ${\rm ch}\,SST(\lambda(a, b, c))_{(k)} = {\rm ch}\,SST(\lambda(a-k, b-k, c+k))_{(0)}$ for $k \geq 0$.
    Then we can prove the lemma by induction on $r = \min\{ a, b \}$ with the Littlewood-Richardson rule. 
    We leave it as an exercise.

\end{proof}

From Lemma \ref{lem:spinor residue character}, we obtain the Schur expansion of a spinor model.
\begin{eqnarray*}
    {\rm ch}\, {\bf T}^\mathfrak{c}(a) &=& t\sum_{c=0}^\infty s_{(a+c, c)'}({\bf x}) \qquad (a \geq 0)\\
    {\rm ch}\, {\bf T}^\mathfrak{b}(a) &=& t\sum_{b, c=0}^\infty s_{(a+b+c, c)'}({\bf x}) \qquad (a \geq 0)\\
    {\rm ch}\, {\bf T}^\mathfrak{d}(a) &=& t\sum_{b, c = 0}^\infty s_{(a+2b+c, c)'}({\bf x}) \qquad (a \geq 1)\\
    {\rm ch}\, {\bf T}^\mathfrak{d}(0) &=& t\sum_{b, c=0}^\infty s_{(2b+2c, 2c)'}({\bf x}) \\
    {\rm ch}\, \overline{\bf T}^\mathfrak{d}(0) &=& t\sum_{b, c=0}^\infty s_{(2b+2c+1, 2c+1)'}({\bf x})
\end{eqnarray*}

For $r \geq 0$, we denote by $e_r({\bf x})$ the $r$-th elementary symmetric function in variables ${\bf x}$, 
and set $e_0({\bf x}) = 1$ and $e_r({\bf x}) = 0$ for $r <0$.
For $r \in \mathbb{Z}$, define
\[ E_r({\bf x}) = \sum_{i=0}^\infty e_i({\bf x}) e_{r+i}({\bf x}). \]
We can show that $E_r ({\bf x}) = E_{-r} ({\bf x})$ for all $r \in \mathbb{Z}$. 
In addition, we can find the $E$-expansion of the character of ${\bf T}^\mathfrak{g}(a)$ using the Jacobi-Trudi formula for Schur functions.
\begin{prop} \label{prop:wt gen ftn of spinor}
    For $a \in \mathbb{Z}_+$ ($a \in \mathbb{N}$ when $\mathfrak{g} = \mathfrak{d}$), the followings hold.
    \begin{eqnarray*}
        {\rm ch}\, {\bf T}^\mathfrak{c}(a) &=& t(E_a({\bf x}) - E_{a+2}({\bf x})) \\
        {\rm ch}\, {\bf T}^\mathfrak{b}(a) &=& t(E_a({\bf x}) + E_{a+1}({\bf x})) \\
        {\rm ch}\, {\bf T}^\mathfrak{d}(a) &=& tE_a({\bf x}) \\
        {\rm ch}\, {\bf T}^\mathfrak{d}(0) + {\rm ch}\, \overline{\bf T}^\mathfrak{d}(0) &=& tE_0({\bf x}) \\
        {\rm ch}\, {\bf T}^\mathfrak{d}(0) - {\rm ch}\, \overline{\bf T}^\mathfrak{d}(0) &=& t\left( \sum_{i=0}^\infty e_i({\bf x}) \right) \left( \sum_{i=0}^\infty (-1)^i e_i({\bf x}) \right)
    \end{eqnarray*}
\end{prop}

\subsection{Jacobi-Trudi type character formulas} \label{subsec:Jacobi-Trudi}
For $r \in \mathbb{Z}$, denote
\begin{eqnarray*}
    e_r({\bf x}_n^{\pm 1}) &=& e_r(x_1, \dots, x_n, x_1^{-1}, \dots, x_n^{-1}), \\
    e_r({\bf x}_n^{\pm 1}, 1) &=& e_r(x_1, \dots, x_n, x_1^{-1}, \dots, x_n^{-1}, 1),
\end{eqnarray*}
where both are Laurent polynomials in $n$ variables, and set
\[ e_r'({\bf x}_n^{\pm 1}) = e_r({\bf x}_n^{\pm 1}) - e_{r-2}({\bf x}_n^{\pm 1}). \]
On the other hand, for $\diamondsuit \in \{ \,\cdot\,, \,'\, \}$ ($\cdot$ means the no marking) 
and $\mu \in \mathcal{P}_n$, we set
\begin{eqnarray*}
    \sigma_\mu^\diamondsuit({\bf x}) &=& \det(e_{(\mu_i-i+1)+(j-1)}^\diamondsuit({\bf x}) + \delta(j \neq 1)e_{(\mu_i-i+1)-(j-1)}^\diamondsuit({\bf x}))_{i, j = 1, \dots, n} \\
        &=& \begin{vmatrix}
        e_{\mu_1}^\diamondsuit & e_{\mu_1 + 1}^\diamondsuit + e_{\mu_1 - 1}^\diamondsuit & \cdots & e_{\mu_1 + (n-1)}^\diamondsuit + e_{\mu_1 - (n-1)}^\diamondsuit \\
        e_{\mu_2-1}^\diamondsuit & e_{(\mu_2-1) + 1}^\diamondsuit + e_{(\mu_2-1) - 1}^\diamondsuit & \cdots & e_{(\mu_2-1) + (n-1)}^\diamondsuit + e_{(\mu_2-1) - (n-1)}^\diamondsuit \\
        \vdots & \vdots & \ddots & \vdots \\
        e_{\mu_n - n+1}^\diamondsuit & e_{(\mu_n - n+1) + 1}^\diamondsuit + e_{(\mu_n - n+1) - 1}^\diamondsuit & \cdots & e_{(\mu_n - n+1) + (n-1)}^\diamondsuit + e_{(\mu_n - n+1) - (n-1)}^\diamondsuit
    \end{vmatrix},
\end{eqnarray*}
where $\delta(P) = 1$ if a condition $P$ is true and $\delta(P) = 0$ otherwise.
Then the character of $B_n^\mathfrak{g}(\varpi_\lambda)$, which is just $B_n(\varpi_\lambda)$ 
with a temporary recording of $\mathfrak{g}$, is given in \cite{FH91,KT87} as follows:
for $\lambda \in \mathcal{P}_n$,
\begin{eqnarray*}
    {\rm ch}\,B_n^\mathfrak{c}(\varpi_\lambda) &=& \sigma_{\lambda'}'({\bf x}_n^{\pm 1}) \\
    {\rm ch}\,B_n^\mathfrak{b}(\varpi_\lambda) &=& \sigma_{\lambda'}({\bf x}_n^{\pm 1}, 1) \\
    {\rm ch}\,B_n^\mathfrak{d} (\varpi_\lambda) &=&
    \begin{cases}
        \sigma_{\lambda'}({\bf x}_n^{\pm 1}) & \mbox{if } \ell(\lambda) < n, \\
        \displaystyle \frac{1}{2} \sigma_{\lambda'}({\bf x}_n^{\pm 1}) + \frac{1}{2} \sigma_{(\lambda-(1^n))'}'({\bf x}_n^{\pm 1}) \prod_{i=1}^n (x_i-x_i^{-1}) & \mbox{if } \ell(\lambda) = n, 
    \end{cases}
\end{eqnarray*}
and for $\lambda \in \mathcal{GP}_n^-$,
\[ {\rm ch}\,B_n^\mathfrak{d} (\varpi_\lambda) = \frac{1}{2} \sigma_{\mu'}({\bf x}_n^{\pm 1}) - \frac{1}{2} \sigma_{(\mu-(1^n))'}'({\bf x}_n^{\pm 1}) \prod_{i=1}^n (x_i-x_i^{-1}), \]
where $\mu = (\lambda_1, \dots, \lambda_{n-1}, |\lambda_n|) \in \mathcal{P}_n$.

On the other hand, we can easily observe the following relations for $r \in \mathbb{Z}$.
\begin{eqnarray*}
    e_{n-r}({\bf x}_n^{\pm 1}) &=& e_{n+r}({\bf x}_n^{\pm 1}) \\
    e_r({\bf x}_n^{\pm 1}, 1) &=& e_r({\bf x}_n^{\pm 1}) + e_{r-1}({\bf x}_n^{\pm 1})
\end{eqnarray*}
For $r \in \mathbb{Z}$, define $E_r'({\bf x}) = E_r({\bf x}) - E_{r+2}({\bf x})$ and $E_r''({\bf x}) = E_r({\bf x}) + E_{r+1}({\bf x})$. 
Then we obtain the following equations for all $r \in \mathbb{Z}$.
\begin{eqnarray}
    e_{n-r}({\bf x}_n^{\pm 1}) &=& (x_1 \dots x_n)^{-1} E_r({\bf x}_n) \label{eqn:Dn elementary symmetric} \\
    e_{n-r}'({\bf x}_n^{\pm 1}) &=& (x_1 \dots x_n)^{-1} E_r'({\bf x}_n) \label{eqn:Cn elementary symmetric} \\ 
    e_{n-r}({\bf x}_n^{\pm 1}, 1) &=& (x_1 \dots x_n)^{-1} E_r''({\bf x}_n) \label{eqn:Bn elementary symmetric}
\end{eqnarray}

\begin{df}[\cite{LZ06}]
    For $\diamondsuit \in \{ \,\cdot\,, \,'\,,\, ''\, \}$ and $(\lambda, \ell) \in \mathscr{P}(G)$, denote
    \begin{eqnarray*}
        \Sigma_{(\lambda, \ell)}^\diamondsuit({\bf x}) &=& \det(E_{(\lambda_{\ell-i+1}+i-1)+(j-1)}^\diamondsuit({\bf x}) + \delta(j \neq 1)E_{(\lambda_{\ell-i+1}+i-1)-(j-1)}^\diamondsuit({\bf x}))_{i, j = 1, \dots, \ell}\\
            &=& \begin{vmatrix}
            E_{\lambda_\ell}^\diamondsuit & E_{\lambda_\ell + 1}^\diamondsuit + E_{\lambda_\ell - 1}^\diamondsuit & \cdots & E_{\lambda_\ell + (\ell-1)}^\diamondsuit + E_{\lambda_\ell - (\ell-1)}^\diamondsuit \\
            E_{\lambda_{\ell-1}+1}^\diamondsuit & E_{(\lambda_{\ell-1}+1) + 1}^\diamondsuit + E_{(\lambda_{\ell-1}+1) - 1}^\diamondsuit & \cdots & E_{(\lambda_{\ell-1}+1) + (\ell-1)}^\diamondsuit + E_{(\lambda_{\ell-1}+1) - (\ell-1)}^\diamondsuit \\
            \vdots & \vdots & \ddots & \vdots \\
            E_{\lambda_1 + \ell-1}^\diamondsuit & E_{(\lambda_1 + \ell-1) + 1}^\diamondsuit + E_{(\lambda_1 + \ell-1) - 1}^\diamondsuit & \cdots & E_{(\lambda_1 + \ell-1) + (\ell-1)}^\diamondsuit + E_{(\lambda_1 + \ell-1) - (\ell-1)}^\diamondsuit
        \end{vmatrix},
    \end{eqnarray*}
    and define $S^\mathfrak{g}_{(\lambda, \ell)}({\bf x})$ by
    \begin{eqnarray*}
        S_{(\lambda, \ell)}^\mathfrak{c}({\bf x}) &=& \Sigma_{(\lambda, \ell)}'({\bf x}), \\
        S_{(\lambda, \ell)}^\mathfrak{b}({\bf x}) &=& \Sigma_{(\lambda, \ell)}''({\bf x}), \\
        S_{(\lambda, \ell)}^\mathfrak{d}({\bf x}) &=& \begin{cases} 
            \Sigma_{(\lambda, \ell)}({\bf x}) & \mbox{if } t = \ell, \\
            \displaystyle \frac{1}{2} \Sigma_{(\lambda, \ell)}({\bf x}) + \frac{1}{2} \left( \sum_{i=0}^\infty e_i({\bf x}) \right) \left( \sum_{i=0}^\infty (-1)^i e_i({\bf x}) \right) \Sigma_{(\lambda, \ell-1)}'({\bf x}) & \mbox{if } t < \ell, \\
            \displaystyle \frac{1}{2} \Sigma_{(\mu, \ell)}({\bf x}) - \frac{1}{2} \left( \sum_{i=0}^\infty e_i({\bf x}) \right) \left( \sum_{i=0}^\infty (-1)^i e_i({\bf x}) \right) \Sigma_{(\mu, \ell-1)}'({\bf x}) & \mbox{if } t > \ell,
        \end{cases}
    \end{eqnarray*}
    where $t = \ell(\lambda)$ and $\mu = (\lambda_1, \dots, \lambda_{2\ell-t})$. Note that $(\mu, \ell) \in \mathscr{P}(G)$ with $\ell(\mu) < \ell$.
\end{df}

\begin{prop} \label{prop:dominant wt character}
    We have the following character identity.
    \[ {\rm ch}\,{\bf T}^\mathfrak{g} (\lambda, \ell) = t^\ell\, S_{(\lambda, \ell)}^\mathfrak{g}({\bf x}) \]
\end{prop}
\begin{proof}
    By Proposition \ref{prop:bij spinor to KN} and equations \eqref{eqn:Dn elementary symmetric}, \eqref{eqn:Cn elementary symmetric}, and \eqref{eqn:Bn elementary symmetric}, we directly show that
    \[ {\rm ch}\,{\bf T}_n^\mathfrak{g}(\lambda, \ell) = {\rm ch}\,B_n(\varpi_{\rho_n(\lambda, \ell)}) = (x_1 \dots x_n)^{-\ell} S^\mathfrak{g}_{(\lambda, \ell)}({\bf x}_n) \]
    for $n \in \mathbb{Z}_+$ and $(\lambda, \ell) \in \mathscr{P}(G)$.
    If we set ${\bf T}_{(n)}^\mathfrak{g}(\lambda, \ell) = {\bf T}_n^\mathfrak{g}(\lambda, \ell) \backslash {\bf T}_{n-1}^\mathfrak{g}(\lambda, \ell)$,
    where ${\bf T}_{-1}^\mathfrak{g}(\lambda, \ell) = \emptyset$, then we have
    \[ {\rm ch}\,{\bf T}^\mathfrak{g}(\lambda, \ell) = \sum_n {\rm ch}\,{\bf T}^\mathfrak{g}_{(n)}(\lambda, \ell)
    = \lim_{n \to \infty} {\rm ch}\,{\bf T}^\mathfrak{g}_n(\lambda, \ell). \]
    Since the weight lattice $P$ is the (direct) limit of the weight lattice associated with $\mathfrak{g}_n$, 
    taking the limit of symmetric polynomials (or characters) is understood as obtaining corresponding symmetric functions. 
    In this context, the term $(x_1 \dots x_n)^{-\ell}$ converges to $t^\ell$ and $S^\mathfrak{g}_{(\lambda, \ell)}({\bf x}_n)$ converges to $S^\mathfrak{g}_{(\lambda, \ell)}({\bf x})$, 
    which completes the proof.
\end{proof}

\section{The Grothendieck ring} \label{sec:Grothendieck}
\subsection{The Grothendieck ring} \label{subsec:Grothendieck background}
Let $\mathcal{C}$ be the category of $\mathfrak{g}_\infty$-crystals whose objects $B$ satisfies 
\begin{enumerate}
    \item there exists a finite subset $S \subseteq \mathcal{P}$ such that each connected component of $B$ is isomorphic to 
    $B(\varpi_\mu)$ or $B(\varpi_\mu) \otimes B(\Pi^\mathfrak{g}(\lambda, \ell))$ for some $\mu \in S$ and $(\lambda, \ell) \in \mathscr{P}(G)$.
    \item for $\Lambda \in P$ with ${\rm lv}(\Lambda) = n \geq 0$, the number of connected components of $B$ isomorphic to $B(\Lambda)$ is finite.
\end{enumerate}
The morphisms in $\mathcal{C}$ are taken to be morphisms of $\mathfrak{g}_\infty$-crystals. 
Note that $B \in \mathcal{C}$ may contain infinitely many connected components.

Let $\mathcal{K} = \mathcal{K}(\mathcal{C})$ be the Grothendieck group of $\mathcal{C}$, i.e., the additive group of isomorphism classes $[B]$ for $B \in \mathcal{C}$ 
with the following relation; $[B \oplus B'] = [B]+[B']$ for $B, B' \in \mathcal{C}$. 
It is clear that a complete list of $[B]$ for non-isomorphic connected crystals $B \in \mathcal{C}$ forms a $\mathbb{Z}$-basis of $\mathcal{K}$. 
We define a multiplication on $\mathcal{K}$ by
\[ [B] \cdot [B'] = [B \otimes B']. \]
The following theorem is proved by the similar argument in the proof of \cite[Theorem 5.1]{K11} with isomorphisms \eqref{eqn:zero tensor zero}, \eqref{eqn:posi tensor posi}, \eqref{eqn:zero tensor posi}, and \eqref{eqn:posi tensor zero}.
\begin{thm}
    The category $\mathcal{C}$ is a monoidal category under the tensor product of crystals. 
    In addition, the Grothendieck group $\mathcal{K}$ is an associative $\mathbb{Z}$-algebra with unity $[B(0)]$.
\end{thm}

Let $\mathcal{C}^0$ and $\mathcal{C}^+$ be the full subcategory of $\mathcal{C}$ consisting of objects 
whose connected components are isomorphic to $B(\lambda)$ for some $\lambda \in E$ and $\lambda \in P^+_{\rm int}$, respectively. 
It is known in \cite{K15,K16} that $\mathcal{C}^+$ is a semisimple tensor category.
We set $\mathcal{K}^0$ and $\mathcal{K}^+$ to be the subgroups of $\mathcal{K}$ generated by $[B(\lambda)]$ for some $\lambda \in E$ and $\lambda \in P^+_{\rm int}$, respectively.
By \eqref{eqn:zero tensor zero}, it immediately follows that $\mathcal{K}^0$ is a subalgebra of $\mathcal{K}$ and $\mathcal{K}^0$ is generated by $[B(\varpi_i)]$. 
In addition, by Proposition \ref{prop:tensor product}\,(3) and Proposition \ref{prop:decomp into zero and posi}, 
we have $\mathcal{K} \cong \mathcal{K}^0 \otimes_\mathbb{Z} \mathcal{K}^+$ and 
$\mathcal{K}$ has a $\mathbb{Z}$-basis $\{ [B(\varpi_\mu)] \cdot [B(\Pi^\mathfrak{g}(\lambda, \ell))] : \mu \in \mathcal{P}, (\lambda, \ell) \in \mathscr{P}(G) \}$. 
In particular, $\{ [B(\varpi_\mu)] : \mu \in \mathcal{P} \}$ and $\{ [B(\Pi^\mathfrak{g}(\lambda, \ell))] : (\lambda, \ell) \in \mathscr{P}(G) \}$ 
are $\mathbb{Z}$-bases of $\mathcal{K}^0$ and $\mathcal{K}^+$, respectively.

For $k \geq 0$ and $(\lambda, \ell) \in \mathscr{P}(G)$, set
\[ H_k^\mathfrak{g} = [B(\Pi_k^\mathfrak{g})], \quad H^\mathfrak{g}(\lambda, \ell) = [B(\Pi^\mathfrak{g}(\lambda, \ell))], \]
and $\overline{H}_0^\mathfrak{d} = [B(\overline{\Pi}_0^\mathfrak{d})]$.
Since $\mathcal{C}^+$ is a semisimple tensor category, we obtain the following proposition 
as a corollary of Proposition \ref{prop:wt gen ftn of spinor} and Proposition \ref{prop:dominant wt character}.
\begin{prop} \label{prop:Jacobi-Trudi formula}
    When $\mathfrak{g} = \mathfrak{d}$ and $(\lambda, \ell) \in \mathscr{P}(G)$ with $\ell(\lambda) < \ell$, the following identity holds in $\mathcal{K}^+$.
    \begin{eqnarray*}
        H^\mathfrak{d}(\lambda, \ell) &=& \frac{1}{2} \det( H^\mathfrak{d}_{(\lambda_{\ell-i+1}+i-1)+(j-1)}+ \delta(j \neq 1) H^\mathfrak{d}_{(\lambda_{\ell-i+1}+i-1)-(j-1)})_{i, j = 1, \dots, \ell} \\
        && +\frac{1}{2} (H^\mathfrak{d}_0 - \overline{H}^\mathfrak{d}_0) H^\mathfrak{c}(\lambda, \ell-1)
    \end{eqnarray*}
    When $\mathfrak{g} = \mathfrak{d}$ and $(\lambda, \ell) \in \mathscr{P}(G)$ with $\ell(\lambda) > \ell$, the following identity holds in $\mathcal{K}^+$.
    \begin{eqnarray*}
        H^\mathfrak{d}(\lambda, \ell) &=& \frac{1}{2} \det( H^\mathfrak{d}_{(\mu_{\ell-i+1}+i-1)+(j-1)}+ \delta(j \neq 1) H^\mathfrak{d}_{(\mu_{\ell-i+1}+i-1)-(j-1)})_{i, j =1, \dots, \ell} \\
        && -\frac{1}{2} (H^\mathfrak{d}_0 - \overline{H}^\mathfrak{d}_0) H^\mathfrak{c}(\mu, \ell-1)
    \end{eqnarray*}
    Here, $t = \ell(\lambda)$ and $\mu = (\lambda_1, \dots, \lambda_{2\ell-t}, 0^{t-\ell})$.
    Otherwise, the following identity holds in $\mathcal{K}^+$.
    \[ H^\mathfrak{g}(\lambda, \ell) = \det( H^\mathfrak{g}_{(\lambda_{\ell-i+1}+i-1)+(j-1)}+ \delta(j \neq 1) H^\mathfrak{g}_{(\lambda_{\ell-i+1}+i-1)-(j-1)})_{i, j \in [\ell]} \]
\end{prop}
This proposition tells us that $\mathcal{K}^+$ is a subalgebra of $\mathcal{K}$ and $\mathcal{K}^+$ is generated by $H^\mathfrak{g}_k$ (and $\overline{H}^\mathfrak{d}_0$ when $\mathfrak{g} = \mathfrak{d}$) for $k \geq 0$.
In fact, $H^\mathfrak{c}(\lambda, \ell-1)$ and $H^\mathfrak{c}(\mu, \ell-1)$ appearing in $H^\mathfrak{d}(\lambda, \ell)$ 
are written as a polynomial in $H^\mathfrak{d}_k$ ($k \geq 0$) and $\overline{H}^\mathfrak{d}_0$ since $E_r'({\bf x}) = E_r({\bf x}) - E_{r+2}({\bf x})$.

\subsection{Subalgebras $\mathcal{K}^0$ and $\mathcal{K}^+$} \label{subsec:K0K+}
To explain the subalgebra $\mathcal{K}^+$, we prove the following lemma.
\begin{lem} \label{lem:dominance lemma}
    For $(\lambda, \ell), (\mu, m) \in \mathscr{P}(G)$, denote by $K^{(\mu, m)}_{(\lambda, \ell)}$ the coefficient of $H^\mathfrak{g}(\lambda, \ell)$ 
    in the product $H^\mathfrak{g}_{\mu_1} \cdots H^\mathfrak{g}_{\mu_m}$ if $\ell(\mu) \leq m$, or 
    $H^\mathfrak{g}_{\mu_1} \cdots H^\mathfrak{g}_{\mu_{2m-t}} (\overline{H}^\mathfrak{g}_0)^{t-m}$ if $\ell(\mu) = t > m$. Then
    \begin{enumerate}
        \item $K^{(\mu, m)}_{(\lambda, \ell)} = 0$ if $m \neq \ell$.
        \item for given $(\lambda, \ell) \in \mathscr{P}(G)$, there exist only finitely many $(\mu, m) \in \mathscr{P}(G)$ such that $K^{(\mu, m)}_{(\lambda, \ell)} \neq 0$.
        In particular, we have $K^{(\lambda, \ell)}_{(\lambda, \ell)} = 1$.
    \end{enumerate}
\end{lem}
\begin{proof}
    Since statements for other $\mathfrak{g}_\infty$ are similarly proved, we assume $\mathfrak{g}_\infty = \mathfrak{c}_\infty$. 
    We set ${\rm wt}(\mu, m) = \{ {\rm wt}(b) \,|\, b \in B(\Pi^\mathfrak{c}_{\mu_1}) \otimes \dots \otimes B(\Pi^\mathfrak{c}_{\mu_m}) \}$. 
    Then we know ${\rm wt}(\mu, m) \subseteq \Pi^\mathfrak{c}(\mu, m) - Q^+$. 

    Suppose that $K^{(\mu, m)}_{(\lambda, \ell)} \neq 0$ for some $(\mu, m), (\lambda, \ell) \in \mathscr{P}({\rm Sp})$. 
    By construction, there exists a (highest) weight vector of weight $\Pi^\mathfrak{c}(\lambda, \ell)$ 
    in $B(\Pi^\mathfrak{c}_{\mu_1}) \otimes \dots \otimes B(\Pi^\mathfrak{c}_{\mu_m})$, which implies $\Pi^\mathfrak{c}(\lambda, \ell) \in {\rm wt}(\mu, m)$. 
    Then we have $m = \ell$ since ${\rm lv}(\gamma) = {\rm lv}(\mu)$ for all $\gamma \in {\rm wt}(\mu, m)$, which proves (1).

    Now, we assume $\ell = m$ and suppose $K^{(\mu, \ell)}_{(\lambda, \ell)} \neq 0$ for given $(\lambda, \ell) \in \mathscr{P}(G)$. 
    We know that the set of positive roots is 
    $\{ \epsilon_i - \epsilon_j \,|\, i < j \} \cup \{ -\epsilon_i - \epsilon_j \,|\, i \leq j \}$.
    By the above argument, $\Pi^\mathfrak{c}(\mu, \ell) - \Pi^\mathfrak{c}(\lambda, \ell) = \varpi_{\mu'} - \varpi_{\lambda'} \in Q^+$ 
    and we suppose that
    \begin{equation} \label{eqn:dominance order equation}
        \varpi_{\mu'} = \varpi_{\lambda'} + \alpha_{i_1} + \dots + \alpha_{i_k}
    \end{equation}
    for some $i_1, \dots, i_k \in I$. 
    Suppose $\ell(\mu') > \ell(\lambda')$ (or $\mu_1 > \lambda_1$) and let $i_t$ be a maximal index among $i_1, \dots, i_k$. 
    If we consider the coefficient of $\epsilon_{M+1}$ in \eqref{eqn:dominance order equation}, where $M = \max\{ i_t, \ell(\lambda') \}$, 
    then it is positive on the left-hand side and non-positive on the right-hand side, which is a contradiction. 
    Hence, we obtain $\ell(\mu') \leq \lambda_1$. Since $\ell(\mu) \leq \ell$, we know $\mu \subseteq (\lambda_1^\ell)$ and there are only finitely many choices of such $\mu$.
\end{proof}

Let ${\bf h} = \{ \,{\tt h}_k \,|\, k \in \mathbb{Z}_+ \}$ $(\{ \,{\tt h}_k \,|\, k \in \mathbb{Z}_+ \} \cup \{ \overline{\tt h}_0 \}$ when $\mathfrak{g} = \mathfrak{d})$ be a set of formal commuting variables 
and denote by $\mathbb{Z}[\![{\bf h}]\!]$ the $\mathbb{Z}$-algebra of formal power series in variables ${\bf h}$.
Now, we construct a $\mathbb{Z}$-algebra isomorphism between $\mathcal{K}^+$ and $\mathbb{Z}[\![{\bf h}]\!]$. 
A characterization of $\mathcal{K}^+$ as a $\mathbb{Z}$-algebra is given 
in \cite{K09b} when $\mathfrak{g}_\infty = \mathfrak{a}_{+\infty}$ and in \cite{K11} when $\mathfrak{g}_\infty = \mathfrak{a}_\infty$.
\begin{thm} \label{thm:Grothendieck positive}
    Define $\Phi^+ : \mathbb{Z}[\![{\bf h}]\!] \to \mathcal{K}^+$ to be a $\mathbb{Z}$-algebra homomorphism sending ${\tt h}_k$ to $H_k^\mathfrak{g}$ (and $\overline{\tt h}_0$ to $\overline{H}_0^\mathfrak{d}$). 
    Then $\Phi^+$ is an isomorphism of $\mathbb{Z}$-algebras.
\end{thm}
\begin{proof}
    Take $f = \sum_{(\mu, \ell)} c(\mu, \ell) {\tt h}_{\mu_1} \dots {\tt h}_{\mu_\ell} \in \mathbb{Z}[\![{\bf h}]\!]$. 
    To check the well-definedness of $\Phi^+$, we will show that $\Phi^+(f) = \sum_{(\mu, \ell)} c(\mu, \ell) H_{\mu_1}^\mathfrak{g} \dots H_{\mu_\ell}^\mathfrak{g}$ is actually contained in $\mathcal{K}^+$. 

    Since $\{ H^\mathfrak{g}(\lambda, \ell) | (\lambda, \ell) \in \mathscr{P}(G) \}$ is a $\mathbb{Z}$-basis of $\mathcal{K}^+$, we have
    \[ \sum_{(\mu, \ell)} c(\mu, \ell) H_{\mu_1}^\mathfrak{g} \dots H_{\mu_\ell}^\mathfrak{g} = 
    \sum_{(\lambda, \ell)} \left(\sum_{(\mu, \ell)} c(\mu, \ell) K^{(\mu, \ell)}_{(\lambda, \ell)} \right) H^\mathfrak{g}(\lambda, \ell), \]
    which means that $\Phi^+(f) \in \mathcal{K}^+$ if and only if $\sum c(\mu, \ell) K^{(\mu, \ell)}_{(\lambda, \ell)}$ is finite.
    By Lemma \ref{lem:dominance lemma}\,(2), the sum is finite and $\Phi^+$ is well-defined. 
    This argument could be applied to the case for $\mathfrak{g}_\infty = \mathfrak{d}_\infty$ similarly.

    We can show that $\Phi^+$ is surjective by Proposition \ref{prop:Jacobi-Trudi formula}. 
    In addition, we also can show that $\Phi^+$ is injective by considering a basis expansion of $\mathcal{K}^+$ and Lemma \ref{lem:dominance lemma}\,(2). 
    We leave it as an exercise.
\end{proof}

On the other hand, the isomorphism \eqref{eqn:zero tensor zero} defines a $\mathbb{Z}$-algebra isomorphism between $\mathcal{K}^0$ and ${\rm Sym}$, the ring of symmetric functions.
\begin{prop} \label{prop:Grothendieck zero}
    There exists an algebra isomorphism
    \begin{equation} \label{eqn:K0}
        \Psi^0 \,:\, \mathcal{K}^0 \longrightarrow {\rm Sym}
    \end{equation}
    sending $[B(\varpi_\lambda)]$ to $s_\lambda$ for $\lambda \in \mathcal{P}$.
\end{prop}

\subsection{Realization of the Grothendieck ring} \label{subsec:Grothendieck realization}
By Theorem \ref{thm:Grothendieck positive} and Proposition \ref{prop:Grothendieck zero}, we explain the algebra structure of $\mathcal{K}^0$ and $\mathcal{K}^+$. 
Now, we are ready to explain the algebra structure of $\mathcal{K}$, which is another main result of this paper.
It is essentially based on the following decomposition of $B(\Pi_a^\mathfrak{g}) \otimes B(\varpi_b)$ into connected components.
\begin{prop} \label{prop:tensor decomp}
    For $a \in \mathbb{Z}_+$ ($a \in \mathbb{N}$ when $\mathfrak{g} = \mathfrak{d}$) and $b \in \mathbb{N}$, we have the following decomposition.
    \begin{eqnarray*}
        B(\Pi_a^\mathfrak{c}) \otimes B(\varpi_b) &\cong& \bigoplus_{i=0}^b \bigoplus_{j=0}^{\min\{a, b-i\}} B(\varpi_i) \otimes B(\Pi_{a+b-i-2j}^\mathfrak{c}) \\
        B(\Pi_a^\mathfrak{b}) \otimes B(\varpi_b) &\cong& \bigoplus_{i=0}^b \left\{ B(\varpi_i) \otimes \left( \bigoplus_{j=0}^{\min\{a, b-i\}} B(\Pi_{a+b-i-2j}^\mathfrak{b}) \right. \right. \\
            && \hspace{3cm}\left. \left. \oplus\ \bigoplus_{k=1}^{b-i-a} B(\Pi_{b-i-a-k}^\mathfrak{b})^{\oplus\, \delta(b-i > a)} \right) \right\} \\
        B(\Pi_a^\mathfrak{d}) \otimes B(\varpi_b) &\cong& \bigoplus_{i=0}^b \left\{ B(\varpi_i) \otimes \left( \bigoplus_{j=0}^{\min\{\lfloor \frac{a+b-i}{2} \rfloor, b-i\}} B(\Pi_{a+b-i-2j}^\mathfrak{d}) \right. \right. \\
            && \hspace{3cm}\left. \left. \oplus\ \bigoplus_{k=0}^{\lfloor \frac{b-i-a}{2} \rfloor} B(\overline{\Pi}_{b-i-a-2k}^\mathfrak{d})^{\oplus\, \delta(b-i \geq a)} \right) \right\} \\
        B(\Pi_0^\mathfrak{d}) \otimes B(\varpi_b) &\cong& \bigoplus_{i=0}^b \bigoplus_{j=0}^{\lfloor \frac{b-i}{2} \rfloor} B(\varpi_i) \otimes B(\Pi_{b-i-2j}^\mathfrak{d}) \\
        B(\overline{\Pi}_0^\mathfrak{d}) \otimes B(\varpi_b) &\cong& \bigoplus_{i=0}^b \bigoplus_{j=0}^{\lfloor \frac{b-i}{2} \rfloor} B(\varpi_i) \otimes B(\overline{\Pi}_{b-i-2j}^\mathfrak{d})
    \end{eqnarray*}
    Here, we put $\overline{\Pi}_i^\mathfrak{d} = \Pi_i^\mathfrak{d}$ for $i \geq 1$.
\end{prop}
\begin{proof}
    We assume $\mathfrak{g}_\infty = \mathfrak{c}_\infty$ because the same argument can be applied to other $\mathfrak{g}_\infty$. 
    We fix a sufficiently large $n$. By Proposition \ref{prop:bij spinor to KN} and Theorem \ref{thm:level zero realization}, we have
    $B_n(\Pi_a^\mathfrak{c}) \otimes B_n(\varpi_b) \cong {\bf KN}_n^\mathfrak{c}((1^{n-a})) \otimes {\bf KN}_n^\mathfrak{c}((1^b))$. 
    It has the following decomposition into connected components (cf. \cite{Na93}).
    \[ {\bf KN}_n^\mathfrak{c}((1^{n-a})) \otimes {\bf KN}_n^\mathfrak{c}((1^b)) \,=\, \bigsqcup_{i=0}^b \bigsqcup_{j=0}^{\min\{a, b-i\}} {\bf KN}_n^\mathfrak{c}((n-a-b+i+2j, i)') \]
    We remark that this decomposition is explicitly described using KN tableaux (cf. \cite{Le02, Le03}) 
    by enumerating all highest weight vectors of ${\bf KN}_n^\mathfrak{c}((1^{n-a})) \otimes {\bf KN}_n^\mathfrak{c}((1^b))$.

    Recall $(\cdot)^{\tt ad} : {\bf T}_n^\mathfrak{c}(a) \to {\bf KN}_n^\mathfrak{c}((1^{n-a}))$ (cf. \eqref{eqn:bij spinor to KN}) 
    is an isomorphism of $\mathfrak{g}_n$-crystals for $0 \leq a \leq n$.
    Let $T_{ij} \in {\bf T}^\mathfrak{c}(a)$ and $S_{ij} \in {\bf KN}^\mathfrak{c}((1^b))$ be such that, as $\mathfrak{g}_n$-crystals, 
    $T_{ij}^{\tt ad} \otimes S_{ij}$ is a highest weight vector and 
    $C(T_{ij}^{\tt ad} \otimes S_{ij})$ is isomorphic to ${\bf KN}_n^\mathfrak{c}((n-a-b+i+2j, i)')$.
    Then Proposition \ref{prop:tensor product}\,(4) suggests that 
    a highest weight vector of $B_n(\Pi_a^\mathfrak{c}) \otimes B_n(\varpi_b)$ as a $\mathfrak{g}_n$-crystal 
    is an extremal weight vector of $B(\Pi_a^\mathfrak{c}) \otimes B(\varpi_b)$ as a $\mathfrak{g}_\infty$-crystal.
    We can easily check that ${\rm wt}(T_{ij} \otimes S_{ij}) = \Pi^\mathfrak{c}_{a+b-i-2j} + \varpi_i$. 
    By Proposition \ref{prop:decomp into zero and posi}, we have 
    \[ B(\Pi^\mathfrak{c}_{a+b-i-2j} + \varpi_i) \cong B(\varpi_i) \otimes B(\Pi^\mathfrak{c}_{a+b-i-2j}), \]
    which completes the proof.
\end{proof}
\begin{rem}
    The tensor decomposition in Proposition \ref{prop:tensor decomp} may not be multiplicity-free. 
    For example, $B(\Pi_1^\mathfrak{d}) \otimes B(\varpi_2)$ has two distinct connected components which are isomorphic to $B(\Pi_1^\mathfrak{d})$.
\end{rem}

We introduce another set $\{ {\tt z}_i \,|\, i \in \mathbb{N} \}$ of formal commuting variables. 
Define $\mathcal{A}_0 = \mathbb{Z}[\![{\bf h}]\!]$, $\mathcal{A}_n = \mathcal{A}_0[{\tt z}_1, \dots, {\tt z}_n]$ for $n \in \mathbb{N}$, 
and $\mathcal{A} = \sum_{n \geq 0} \mathcal{A}_n$. 
We inductively define a $\mathbb{Z}$-algebra structure on $\mathcal{A}$ as follows.
\begin{itemize}
    \item The multiplication on $\mathcal{A}_0$ is the usual multiplication.
    \item Suppose that the multiplication on $\mathcal{A}_{n-1}$ is defined. 
    Define $a{\tt z}_n = {\tt z}_na + \delta_n(a)$ for $a \in \mathcal{A}_{n-1}$, where $\delta_n$ is a derivation on $\mathcal{A}_{n-1}$ such that
    \begin{eqnarray*}
        \mathfrak{c}_\infty && \begin{cases}
            \delta_n({\tt z}_k) = 0 \qquad\qquad (1 \leq k \leq n-1) \\
            \displaystyle\delta_n({\tt h}_a) = \sum_{i=0}^{n-1} \sum_{j=0}^{\min\{a, n-i\}} {\tt z}_i {\tt h}_{a+n-i-2j} & (a \in \mathbb{Z}_+)
        \end{cases} \\
        \mathfrak{b}_\infty && \begin{cases}
            \delta_n({\tt z}_k) = 0 \qquad\qquad (1 \leq k \leq n-1) \\
            \displaystyle \delta_n({\tt h}_a) = \sum_{i=0}^{n-1} {\tt z}_i \left( \sum_{j=0}^{\min\{a, b-i\}} {\tt h}_{a+b-i-2j} + \delta(b-i > a)\sum_{k=1}^{b-i-a} {\tt h}_{b-i-a-k} \right) & (a \in \mathbb{Z}_+)
        \end{cases} \\
        \mathfrak{d}_\infty && \begin{cases}
            \delta_n({\tt z}_k) = 0 \qquad \qquad (1 \leq k \leq n-1) \\
            \displaystyle \delta_n({\tt h}_0) = \sum_{i=0}^{n-1} \bigoplus_{j=0}^{\lfloor \frac{n-i}{2} \rfloor} {\tt z}_i {\tt h}_{b-i-2j}, 
            \quad \delta_n(\overline{\tt h}_0) = \sum_{i=0}^{n-1} \bigoplus_{j=0}^{\lfloor \frac{n-i}{2} \rfloor} {\tt z}_i \overline{\tt h}_{b-i-2j} \\
            \displaystyle \delta_n({\tt h}_a) = \sum_{i=0}^{n-1} {\tt z}_i \left( \sum_{j=0}^{\min\{\lfloor \frac{a+b-i}{2} \rfloor, b-i\}} {\tt h}_{a+b-i-2j} + \delta(b-i \geq a) \sum_{k=0}^{\lfloor \frac{b-i-a}{2} \rfloor} \overline{\tt h}_{b-i-a-2k} \right) & (a \in \mathbb{N}),
        \end{cases}
    \end{eqnarray*}
\end{itemize}
where $\overline{\tt h}_a = {\tt h}_a$ for $a \geq 1$.

When $\mathfrak{g} = \mathfrak{a}_\infty$, the characterization of $\mathcal{K}$ is given in \cite[Proposition 5.6]{K11}. 
By applying a similar argument in \cite[Proposition 5.6]{K11} with Proposition \ref{prop:tensor decomp}, 
we explicitly establish an algebra structure of $\mathcal{K}$.
\begin{thm} \label{thm:whole Grothendieck}
    The assignment sending $[B(\Pi^\mathfrak{g}_a)]$ to ${\tt h}_a$ ($[B(\overline{\Pi}^\mathfrak{d}_0)]$ to $\overline{\tt h}_0$) 
    and $[B(\varpi_b)]$ to ${\tt z}_b$ defines a $\mathbb{Z}$-algebra isomorphism $\Psi : \mathcal{K} \to \mathcal{A}$. 
    Indeed, we have $\Psi = \Psi^0 \otimes \Psi^+$, where $\Psi^+$ is the inverse map of $\Phi^+$ given in Theorem \ref{thm:Grothendieck positive}.
\end{thm}

{\bf Declarations}

{\bf Ethical Approval} not applicable

{\bf Funding} This research was supported by Basic Science Research Program through the National Research Foundation of Korea(NRF) funded by the Ministry of Education (RS-2023-00241542).

{\bf Availability of data and materials} not applicable

\bibliographystyle{plain}
\bibliography{ext_wt_crystal_inf.bib}

\end{document}